\documentclass[11pt]{amsart}

\newcommand{\affine}{\mathbb{A}}
\newcommand{\QQ}{\mathbb{Q}}
\newcommand{\FF}{\mathbb{F}}
\newcommand{\ZZ}{\mathbb{Z}}
\newcommand{\RR}{\mathbb{R}}

\newcommand{\PP}{\mathbb{P}}
\newcommand{\Ocal}{\mathcal{O}}
\newcommand{\basin}{\mathcal{B}}
\newcommand{\SL}{\operatorname{SL}}
\newcommand{\Div}{\operatorname{Div}}
\newcommand{\conv}{\operatorname{conv}}
\newcommand{\ord}{\operatorname{ord}}
\newcommand{\Supp}{\operatorname{Supp}}
\newcommand{\Gal}{\operatorname{Gal}}

\newcommand{\julia}{\mathcal{K}}

\renewcommand{\phi}{\varphi}
\renewcommand{\epsilon}{\varepsilon}

\newtheorem{theorem}{Theorem}[section]
\newtheorem{lemma}[theorem]{Lemma}
\newtheorem{conjecture}[theorem]{Conjecture}
\newtheorem{prop}[theorem]{Proposition}

\theoremstyle{definition}

\title{Canonical Heights  for H\'{e}non maps}
\author{Patrick Ingram}
\date{\today}
\address{Department of Mathematics, Colorado State University}
\email{pingram@math.colostate.edu}
\thanks{This work was supported in part by a Discovery Grant from NSERC of Canada.}

\begin{document}

\maketitle

\begin{abstract}
We consider the arithmetic of H\'{e}non maps \[\phi(x, y)=(ay, x+f(y))\] defined over number fields and function fields, usually with the restriction $a=1$.  
We prove a result on the variation of Kawaguchi's canonical height in families of H\'{e}non maps, and derive from this a specialization theorem, showing that the set of parameters above which a given non-periodic point becomes periodic is a set of bounded height.  Proving this involves showing that the only points of canonical height zero for a H\'{e}non map over a function field are those which are periodic (in the non-isotrivial case).  In the case of quadratic H\'{e}non maps $\phi(x, y)=(y, x+y^2+b)$, we obtain a stronger result, bounding the canonical height below by a quantity which grows linearly in the height of $b$, once the number of places of bad reduction is fixed.  Finally, we propose a conjecture regarding $\QQ$-rational periodic points for quadratic H\'{e}non maps defined over $\QQ$, namely that they can only have period 1, 2, 3, 4, 6, or 8.  We check this conjecture for the first million values of $b\in\QQ$, ordered by height.
\end{abstract}

\section{Introduction}
The study of the arithmetic properties of H\'{e}non maps was initiated by Silverman \cite{silverman}, who showed that if $K$ is a number field, $a\in K^*$, and $b\in K$, then the periodic points  for the H\'{e}non map
\[\phi(x, y)=(ay, x+y^2+b)\]
are contained in a set of bounded height.  In particular, such a map has only finitely many $K$-rational periodic points.  This result was subsequently extended and generalized by Denis \cite{denis}, Kawaguchi \cite{kawaguchi}, and Marcello \cite{marcello}.  The purpose of this note is to explore further the canonical heights associated to H\'{e}non maps, i.e., maps of the form
\[\phi(x, y)=(ay, x+f(y)),\]
where $f$ is a polynomial of degree at least 2, defined over a number field or function field.
In particular, we prove analogues of several results known for dynamics of polynomials of one variable.  It should be noted, however, that the one-variable case is considerably simpler; any rational self-map of $\affine^1$ extends to an endomorphism of $\PP^1$, and hence the apparatus of Weil's height machine may be applied.  H\'{e}non maps, on the other hand, do not extend to endomorphisms of a projective surface, and so standard results on dynamics of projective varieties need not apply.  We avoid these problems by taking a somewhat more explicit approach, constructing explicit local height functions.

Kawaguchi~\cite{kawaguchi} constructed canonical heights associated to polynomial automorphisms of $\affine^2$ which, in the case of H\'{e}non maps, are given by
\begin{gather*}
\hat{h}_{\phi}^+(P)=\lim_{N\to\infty}d^{-N}h(\phi^N(P))\\
\hat{h}_{\phi}^-(P)=\lim_{N\to\infty}d^{-N}h(\phi^{-N}(P))\\
\hat{h}_\phi(P)=\hat{h}^+_\phi(P)+\hat{h}_\phi^-(P).
\end{gather*}
Our first result shows that this canonical height varies regularly in families.
\begin{theorem}\label{th:tate}
Let $C$ be a smooth, projective curve over a number field $K$, and let $\phi(x, y)=(ay, x+f(y))$ be defined over $F=K(C)$. Then if $P\in \affine^2(F)$, there are divisors $D_+, D_-\in\Div(C)\otimes \QQ$, depending on $\phi$ and $P$, such that
\[\hat{h}_{\phi_t}^+(P_t)=h_{D_+}(t)+O(1)\]
and
\[\hat{h}_{\phi_t}^-(P_t)=h_{D_-}(t)+O(1).\]
\end{theorem}
As a corollary to this result, we note that for any degree-one height $h$ on $C$, we have
\begin{equation}\label{eq:variation simple}\hat{h}_{\phi_t}(P_t)=\hat{h}_\phi(P)h(t)+\epsilon(t),\end{equation}
where $\epsilon(t)=O(1)$ if $C=\PP^1$, and $\epsilon(t)=O(h(t)^{1/2})$ in general.

This result is analogous to a result of the author \cite{variation}, which strengthened more general estimates of Call and Silverman \cite{call-silv} in the case of polynomials of one variable.  One application of an estimate of the form \eqref{eq:variation simple} is in determining which specializations of a one-parameter family land in periodic cycles (note that, since H\'{e}non maps are automorphisms, orbits are either periodic or infinite in both directions).  Since vanishing of the canonical height uniquely identifies periodic points over number fields, we see that the set of periodic specializations of a one-parameter family must be a set of bounded height, unless $\hat{h}_{\phi}(P)=0$ (on the generic fibre).  However, since the Northcott finiteness property does not hold in the context of function fields, it is not obvious when this condition obtains.  We provide an answer for H\'{e}non maps of a certain form, a theorem which is analogous to results of Benedetto \cite{benedetto} and Baker \cite{baker} in the univariate case.  For the purpose of the following result, a \emph{function field} will be any field with a set of non-trivial non-archimedean absolute values satisfying a product formula, a definition which encompasses function fields of smooth varieties over algebraically closed fields of any characteristic.

\begin{theorem}\label{th:benedetto}
Let $K$ be a function field, and let $\phi(x, y)=(y, x+f(y))$ for $f(z)\in K[z]$ of degree at least 2 (note that $a=1$).  Then either $\phi$ is isotrivial or else the set of elements $P\in\affine^2(K)$ with $\hat{h}_\phi(P)=0$ is finite, bounded in size in terms of the number of places of bad reduction for $\phi$.  In particular, if $\phi$ is not isotrivial, then $\hat{h}_\phi(P)=0$ if and only if $P$ is periodic for $\phi$.
\end{theorem}

We define isotriviality below, but in the case of a function field of a variety, it corresponds to the map having constant coefficients after some linear change of variables.  It would, of course, be of considerable interest to obtain a version of Theorem~\ref{th:benedetto} in which one does not assume $a=1$, as it would for several of the results below.

We note that Theorem~\ref{th:benedetto} gives a bound on the number of periodic points for a H\'{e}non map over a function field, which depends only on the degree of $f$ and the number of places of bad reduction.  The proof can be modified to give a similar result over number fields, but in this case the result is already known, due to work of Pezda~\cite{pezda}.

Theorems~\ref{th:tate} and~\ref{th:benedetto} allow us to conclude the following specialization theorem, reminiscent of a result of Silverman for elliptic surfaces.

\begin{theorem}\label{th:silverman}
Let $\phi(x, y)=(y, x+f(y))$ be a non-isotrivial H\'{e}non map defined over the function field of a curve $C$ defined over a number field $K$.  Then either $P\in\affine^2(K(C))$ is periodic for $\phi$, or else
\[\{t\in C(\overline{K}):P_t\text{ is periodic for }\phi_t\}\]
is a set of bounded height.
\end{theorem}

Once one has a bound on the number of points of canonical height zero, it is natural to ask if there is any non-trivial lower bound on the smallest positive values of the canonical heights associated to maps within a given family, in the spirit of conjectures of Lang \cite[VIII Conjecture 9.9]{aec} and Silveman \cite[Conjecture 4.98]{ads}.  It turns out that we can establish such a lower bound for a particular family, depending on the number of places of bad reduction.

\begin{theorem}\label{th:lang}
Let $K$ be a number field or a function field, and let $\phi(x, y)=(y, x+y^2+b)$.  Then for any $s\geq 1$, there exist $B\in\ZZ^{+}$ and $\epsilon>0$ such that if $b\in K$ is $s$-integral, and $P\in \affine^2(K)$, then either $P$ is periodic for $\phi$ of period at most $B$, or else
\[\hat{h}_\phi(P)\geq \epsilon \max\{h(b), 1\}.\]
\end{theorem}

We suspect that, as in \cite{ads_lang}, a simple modification of the proof will give a similar lower bound for the canonical heights associated to the maps $\phi(x, y)=(y, x+y^d+b)$.  It is reasonable to conjecture, of course, that the quantities $B$ and $\epsilon$ in Theorem~\ref{th:lang} can be made absolute.  In particular, one might expect that there is an absolute bound on the size of a periodic cycle for a H\'{e}non map of the form $\phi(x, y)=(y, x+y^2+b)$ over $\QQ$.
We present a precise conjecture here, along the lines of a similar conjecture for univariate quadratic polynomials due to Poonen \cite{poonen}.
\begin{conjecture}\label{conj}
Let $b\in\QQ$ and $\phi(x, y)=(y, x+y^2+b)$.  If $P\in \affine^2(\QQ)$ has period $N$ for $\phi$, then $N\in\{1, 2, 3, 4, 6, 8\}$.
\end{conjecture}

It is possible to construct examples of points of each of these periods, and the only one that presents any computational difficulty is $N=8$.  For example, the map
$\phi(x, y)=(y, x+y^2-9/16)$ has two fixed points, $P_1=(3/4, 3/4)$ and $P_2=(-3/4, -3/4)$, a point $P_3=(3/4, -3/4)$ of period 2, and a point  $P=(1/4, -3/4)$ of period 8.  

  Although it seems likely that Conjecture~\ref{conj} is at least as difficult to prove as Poonen's Conjecture, which remains open, we use techniques similar to those used by Hutz and the author~\cite{i-hutz} (and based on the aforementioned work of Pezda \cite{pezda}) to show that it holds in at least the first million cases.
\begin{prop}\label{prop:comp}
Let $b\in\QQ$ with $H(b)\leq 1000$, and $\phi(x, y)=(y, x+y^2+b)$.  If $P\in \affine^2(\QQ)$ has period $N$ for $\phi$, then $N\in\{1, 2, 3, 4, 6, 8\}$.
\end{prop}
Note that $H$ here is the mutliplicative height, defined by $H(n/m)=\max\{|n|, |m|\}$, for coprime integers $n$ and $m$.

We note one final conjecture, and one partial result, on the specialization of families of H\'{e}non maps.  Theorem~\ref{th:silverman} shows that the set of parameters $t$ at which a family $(\phi, P)$ becomes periodic is a set of bounded height, but it seems likely that this set is still infinite over the algebraic closure of the base field.  If one considers a family $\phi$ over a curve $C$, and two orbits which do not intersect, it seems unlikely that these orbits would coincide on infinitely many fibres.  We posit that the following statement holds, where
\[\Ocal_\phi(P)=\{\phi^N(P):N\in\ZZ\}\]
is the \emph{orbit of $P$ under $\phi$}.
\begin{conjecture}
Suppose that $\phi(x, y)$ is a H\'{e}non map over $F=K(C)$, and that $P, Q\in\affine^2(F)$ have distinct orbits under $\phi$.  Then there exist only finitely many $t\in C(\overline{K})$ such that $\mathcal{O}_{\phi_t}(P_t)=\mathcal{O}_{\phi_t}(Q_t)$.
\end{conjecture}

To give some evidence of this conjecture, we prove the following weak form of the statement for quadratic H\'{e}non maps, where we require infinitely many parameters rational over the ground field, and integral with respect to a certain divisor.

\begin{theorem}\label{th:unlikely}
Let $F=K(C)$, for $C/K$ a curve and $K$ a number field, let $b\in F$ with pole divisor $\eta\in \Div(C)$, let $\phi(x, y)=(y, x+y^2+b)$, and let $P, Q\in \affine^2(F)$ have distinct orbits under $\phi$. For any $s\geq 1$, there  exist only finitely many $t\in C(K)$, $s$-integral with respect to $\eta$, such that $\Ocal_{\phi_t}(P_t)=\Ocal_{\phi_t}(Q_t)$.
\end{theorem}

In Section~\ref{sec:preliminaries} we set out the basic tools, namely local heights, needed for the proofs of the main results.
In Section~\ref{sec:benedetto proof} we prove Theorem~\ref{th:benedetto}, and in Section~\ref{sec:lang proof} we prove Theorem~\ref{th:lang}; the proofs are separate, but rely on similar ideas.  We treat Theorem~\ref{th:tate} in Section~\ref{sec:tate proof}, relying heavily on material from \cite{variation}, and in Section~\ref{sec:applications} we turn our attention to the proofs of Theorems~\ref{th:silverman} and \ref{th:unlikely}.  Finally, Section~\ref{sec:computations} is devoted to describing the computations need to verify Proposition~\ref{prop:comp}, and here we also undertake an initial investigation of the curves  parametrizing quadratic H\'{e}non maps  with a marked point of period $N$.  In the arXiv version of this paper, an appendix contains the Pari/GP code necessary to verify Proposition~\ref{prop:comp}.


\section{Local heights and other preliminaries}\label{sec:preliminaries}

In this section we set out a theory of local heights for H\'{e}non maps.  It should be noted that local heights for regular affine automorphisms have already been considered by Kawaguchi \cite{kawaguchi2}; although the results of \cite{kawaguchi2} are considerably more general than those developed in this section, the special case in which we find ourselves affords a greater level of specificity.

Throughout this section we will assume that $K$ is a field with a valuation $v$, which might be archimedean or non-archimedean.  We will also fix a monic polynomial
\[f(z)=z^d+b_{d-1}z^{d-1}+\cdots+b_0\in K[z],\]
and consider the H\'{e}non map $\phi(x, y)=(ay, x+f(y))$, for some fixed $a\in K^*$.

  If $r\in\RR$, we will set
\[(r)_v=\begin{cases}r & \text{if }v\text{ is archimedean}\\ 1& \text{otherwise}.\end{cases}\]
We take
$\|x, y\|_v=\max\{|x|_v, |y|_v\}$,
and define local canonical heights for $\phi$ by the limits
\begin{gather*}
\hat{\lambda}^{+}_{v, \phi}(P)=\lim_{N\to\infty}d^{-N}\log^+\|\phi^N(P)\|_v\\
\hat{\lambda}^{-}_{v, \phi}(P)=\lim_{N\to\infty}d^{-N}\log^+\|\phi^{-N}(P)\|_v.
\end{gather*}
That these limits exist follows from the work of Kawaguchi \cite{kawaguchi2}, although we prove this again below.  For convenience, we will also set $\hat{\lambda}_{v, \phi}(P)=\hat{\lambda}_{v, \phi}^+(P)+\hat{\lambda}_{v, \phi}^-(P)$.

Our first lemma describes the basic properties of these local height functions.  In order to state the lemma, we set, 
for any monic polynomial $f(z)=z^d+b_{d-1}z^{d-1}+\cdots+b_0$ with coefficients in $K$, 
\[C_{f, v}=\max_{0\leq i<d}\{|b_i|_v^{1/(d-i)}, 1\}.\]
Given $\phi$ as above, we let
\[\basin^+_v(\phi)=\left\{(x, y)\in \affine^2(K):|y|_v>(d+2)_v\max\{|x|_v^{1/d},  C_{f, v}, |a|_v^{1/(d-1)}\}\right\},\]
and
\[\basin^-_v(\phi)=\left\{(x, y)\in \affine^2(K):|a^{-1}x|_v>(d+2)_v\max\{|y|_v^{1/d}, C_{f, v}, |a|_v^{1/(d-1)}\}\right\}.\]

\begin{lemma}\label{lem:local heights}
Let $\phi$, $\basin^+_v(\phi)$, and $\basin^-_v(\phi)$ be as defined above.
\begin{enumerate} 
\item The set $\basin_v^{+}(\phi)$ is closed under the action of $\phi$.
\item The limit defining $\hat{\lambda}_\phi^{+}(P)$ exists, for all $P$, and the function satisfies $\hat{\lambda}_\phi^{+}(\phi(P))=d\hat{\lambda}_\phi^{+}(P)$.
\item For all $P=(x, y)\in\basin^{+}_v(\phi)$, we have
\[\hat{\lambda}^+_{v,\phi}(P)=\log|y|_v+\epsilon^+(b, P, v)\] 
where $\epsilon^+=0$ if $v$ is non-archimedean, and \[-\frac{1}{d-1}\log(d+2)\leq \epsilon^+\leq \frac{1}{d-1}\log \left(\frac{2d+3}{d+2}\right)\] otherwise.
\item We have $\hat{\lambda}_{v, \phi}^+(P)=0$ if and only if there is no $N$ with $\phi^N(P)\in \basin_v^+(\phi)$.
\item The set $\basin_v^{-}(\phi)$ is closed under the action of $\phi^{-1}$.
\item  The limit defining $\hat{\lambda}_\phi^{-}(P)$ exists, for all $P$, and the function satisfies $\hat{\lambda}_\phi^{-}(\phi^{-1}(P))=d\hat{\lambda}_\phi^{-}(P)$.
\item For all $P=(x, y)\in\basin^{-}_v(\phi)$, we have
\[\hat{\lambda}^-_{v,\phi}(P)=\log|x|_v-\frac{d}{d-1}\log|a|_v+\epsilon^-(b, P, v)\] 
where $\epsilon^-=0$ if $v$ is non-archimedean, and \[-\frac{1}{d-1}\log(d+2)\leq \epsilon^-\leq \frac{1}{d-1}\log \left(\frac{2d+3}{d+2}\right)\] otherwise.
\item We have $\hat{\lambda}_{v, \phi}^-(P)=0$ if and only if there is no $N$ with $\phi^{-N}(P)\in \basin_v^-(\phi)$.
\end{enumerate}
\end{lemma}

\begin{proof}
We start with the case that $v\in M_K$ is non-archimedean.
Suppose  that $P=(x, y)\in\basin^{+}_v(\phi)$.  We have
$|y^d|_v> |x|_v$, and $|y^d|_v> |b_iy^i|_v$ for all $i<d$, and so
\[|x+f(y)|_v=|y|_v^d>|ay|_v,\]
and so $\|\phi(P)\|_v=|y|_v^d$.
At the same time,
\[|y|_v^d>|y|_v>\max\{C_{f, v}, |a|_v^{1/(d-1)}\}\]
and
\[|y|_v^d\geq|y|_v^{d/(d-1)}>|ay|_v^{1/(d-1)},\]
and so $\phi(P)\in\basin_v^+(\phi)$.  Thus, $\basin_v^+(\phi)$ is closed under the action of $\phi$, and by induction we see that $P\in \basin_v^+(\phi)$ implies
\[\|\phi^N(P)\|_v=|y|^{d^N}\]
for all $N\geq 1$.
This shows that $\hat{\lambda}^+_{v, \phi}(P)=\log|y(P)|_v$ if $P\in\basin_v^+(\phi)$.  It also follows that the limit defining $\hat{\lambda}^+_{v, \phi}(P)$ exists whenever $\phi^N(P)\in\basin_v^+(\phi)$ for some $N\geq 0$.

If, on the other hand, there is no $N\geq 0$ such that $\phi^N(P)\in\basin_v^+(\phi)$, then write $\phi^N(P)=(x_N, y_N)$.
For each $N\geq 1$ we have
\begin{eqnarray*}
|y_N|_v&\leq& \max\{|x_N|_v^{1/d},  C_{f, v}, |a|_v^{1/(d-1)}\}\\
&= &\max\{|ay_{N-1}|_v^{1/d}, C_{f, v}, |a|_v^{1/(d-1)}\}\\
&\leq &\max\{|y_{N-1}|_v, 1\}^{1/d}\cdot C_{f, v}\cdot \max\{|a|_v^{1/(d-1)}, 1\}\\
&\leq&\max\{|y_0|_v, 1\}^{1/d^N}\cdot C_{f, v}^{1+\frac{1}{d}+\cdots+\frac{1}{d^{N-1}}}\cdot \max\{|a|_v^{1/(d-1)}, 1\}^{1+\frac{1}{d}+\cdots+\frac{1}{d^{N-1}}}\\
&\leq & \max\{|y_0|_v, 1\}^{1/d^N}\cdot C_{f, v}^{d/(d-1)}\cdot \max\{|a|_v^{1/(d-1)}, 1\}^{d/(d-1)}
\end{eqnarray*}
In particular, \[\lim_{N\to\infty}d^{-N}\log^+|y_N|=0,\] and since $x_N=ay_{N-1}$, we have the same for $x_N$.  Consequently, under the hypothesis that there is no $N\geq 0$ with $\phi^N(P)\in \basin^+_v(\phi)$, we have $\hat{\lambda}^+_{v, \phi}(P)=0$. 

In the case of an archimedean valuation $v$, the arguments are similar.  In particular, if $P=(x, y)\in\basin^{+}_v(\phi)$, then $|b_iy^i|_v\leq \frac{1}{d+2}|y^d|$ for all $i<d$, and $|x|_v\leq \frac{1}{d+2}|y^d|_v$, and hence
\[\left(1-\frac{d+1}{d+2}\right)|y|_v^d\leq |x+f(y)|_v\leq \left(1+\frac{d+1}{d+2}\right)|y|_v^d.\]
We also have
\[|ay|_v\leq \left(\frac{1}{d+2}\right)^{d-1}|y|_v^d\leq |x+f(y)|_v,\]
and
\[|y|_v<\left(\frac{1}{d+2}\right)|y|_v^2\leq |x+f(y)|_v,\]
from which we conclude both that $\basin_v^+(\phi)$ is closed under the action of $\phi$, and that $\|\phi(x, y)\|_v=|x+f(y)|_v$.  It follows that, for $P\in\basin_v^+(\phi)$, we have
\[\left(\frac{1}{d+2}\right)|y_P|_v^d\leq \left|y_{\phi(P)}\right|_v=\|\phi(P)\|_v\leq \left(\frac{2d+3}{d+2}\right)|y_P|_v^d.\]
By induction, and taking logarithms and limits,  we have
\[-\frac{1}{d-1}\log(d+2)\leq \hat{\lambda}_{v, \phi}^+(P)-\log|y_P|_v\leq \frac{1}{d-1}\log\left(\frac{2d+3}{d+2}\right).\]

The proofs of the corresponding results  for $\hat{\lambda}^-_{v, \phi}$ are essentially the same.
\end{proof}

In considering the dynamics of polynomial actions on $\affine^1$, it is customary to consider them up to change of variables, that is, up to conjugation by an affine-linear map $z\mapsto \alpha z+\beta$.  It is natural to adopt a similar sense of equivalence in this context.  We will say that two polynomial maps \[\phi_1, \phi_2:\affine^2\to\affine^2,\] defined over a field $K$, are \emph{affine conjugate} if and only if there is a map
\[\psi(x, y)=\left(\alpha x+\beta y+s, \gamma x+\delta y + t\right)\]
with coefficients in $\overline{K}$, such that $\alpha\delta-\beta\gamma\neq 0$ and such that $\phi_2=\psi^{-1}\circ\phi_1\circ\psi$.  Considering maps up to such conjugacy shows that some of the apparent restrictions of the form of map we have chosen, for example the assumption that $f(y)$ is monic, are not genuine restrictions.  In particular, if $f(y)$ is not monic, an affine-linear change of variables transforms the H\'{e}non map to one in which the corresponding polynomial is monic, and so there is no loss of generality inherent in studying only this case.

Note that affine conjugacy is certainly a natural sense of equivalence to use in studying the canonical height.  The affine map $\psi^{-1}$ always extends to an automorphism of $\PP^2$, and so (if $K$ is a number field or function field), we have
\[h(\psi^{-1}(P))=h(P)+O_\psi(1).\]
It follows, if $\phi_1$ and $\phi_2$ are related as above, that
\begin{multline*}
\hat{h}^+_{\phi_2}(P)=\lim_{N\to\infty}d^{-N}h\left(\phi_2^N(P)\right)\\=\lim_{N\to\infty}d^{-N}\left(h\left(\phi_1^N(\psi(P))\right)+O(1)\right)=\hat{h}^+_{\phi_1}(\psi(P))
\end{multline*}
for all $P$, and similarly for $\hat{h}^-_{\phi_2}$.  One may also easily compute the effect of an affine-linear transformation on the local height functions $\hat{\lambda}^\pm_{v, \phi}$.

Since we are considering maps up to this equivalence, it is worth noting which maps of our chosen form are affine-conjugate to one another.  To this end, one easily checks that the following lemma holds.
\begin{lemma}\label{lem:affine conjugacy}
The H\'{e}non map $\phi_1(x, y)=(ay, x+f(y))$ is affine-conjugate to another H\'{e}non map $\phi_2$ if and only if the latter has the form
\[\phi_2(x, y)=\left(ay, x+\delta^{-1}f(\delta y+t)+\delta^{-1}(a-1)t\right),\]
where $\delta^{d-1}=1$.  In particular, if $a=1$, then the affine-conjugacy class of the H\'{e}non map determined by $a$ and $f$ is invariant under a precomposition of $f$ with a translation.
\end{lemma}

\section{Filled Julia sets, and the proof of Theorem~\ref{th:benedetto}}\label{sec:benedetto proof}

We proceed now with the proof of Theorem~\ref{th:benedetto}, an analogue of a result of Benedetto~\cite{benedetto}.  The argument is similar in spirit to the proof of Theorem~\ref{th:lang}, below, but the details diverge somewhat, and so we have not attempted to unify the exposition.  
Throughout this section, we consider a H\'{e}non map of the form
\[\phi(x, y)=(y, x+f(y)),\]
defined over a field $F$.   Ultimately, the field $F$ will be a function field in the sense described above, but until Lemma~\ref{lem:benedetto isotrivial}, one might consider it simply to be a field equipped with one or more non-archimedean absolute values.  We will also suppose that every one of these absolute values has been extended in some way to $\overline{F}$.

We will assume, throughout, that $d=\deg(f)\geq 3$, an assumption used in the proof of Lemma~\ref{lem:benedetto main lemma}.  This is a minor assumption, though, as any H\'{e}non map with $f$ quadratic is affine-conjugate to one of the maps to which Theorem~\ref{th:lang} applies, by Lemma~\ref{lem:affine conjugacy}, and this change of variables increases the number of places of bad reduction by at most the number of places above 2.

 We  define, for any polynomial $g(z)\in F[z]$ and any valuation $v\in M_F$,
\[\rho_v(g)=\max\{1, |\zeta_1-\zeta_2|_v:g(\zeta_1)=g(\zeta_2)=0\},\]
 and \[\Delta_g=\{(\zeta_1, \zeta_2):g(\zeta_1)=g(\zeta_2)=0\}.\]
 It is worth noting that for $\phi(x, y)=(y, x+f(y))$, the set $\Delta_f\subseteq\affine^2(\overline{F})$ is precisely the set of points of period dividing 2, with $\phi$ acting on $\Delta_f$ as reflection across the diagonal.
 
Given any point $Q\in\affine^2(\overline{F}_v)$, we will define the \emph{$v$-adic closed disk of radius $r$ about $Q$} by
\[\overline{D_v(Q; r)}=\{P\in\affine^2(\overline{F}_v):\|P-Q\|_v\leq r\}.\]
Finally, we define the \emph{$v$-adic filled Julia set} of $\phi$ by
\[\julia_{v, \phi}=\left\{P\in\affine^2(\overline{F}_v):\|\phi^N(P)\|_v\text{ is bounded as }N\to\pm\infty\right\}.\]
Note that, by the proof of Lemma~\ref{lem:local heights}, $\julia_{v, \phi}$ coincides precisely with the common vanishing of $\hat{\lambda}_{v, \phi}^+$ and $\hat{\lambda}_{v, \phi}^-$.

\begin{lemma}\label{lem:benedetto filled julia}
For any $\phi$, and any place $v$, we have
\[\julia_{v, \phi}\subseteq\bigcup_{Q\in \Delta_f}\overline{D_v(Q; 1)}\subseteq \overline{D_v(Q'; \rho_v(f))},\] for any $Q'\in\Delta_f$.
\end{lemma}

\begin{proof}
The second containment is simply the ultrametric inequality.  Suppose that $P=(x, y)$ is not contained in $\overline{D_v(Q'; \rho_v(f))}$, for our given point $Q'\in\Delta_f$.  We will show that $P\not\in\julia_{v, \phi}$ by an argument very similar to, be not exactly following from, the proof of Lemma~\ref{lem:local heights}.  We will first assume that $|x-x_{Q'}|_v\leq |y-y_{Q'}|_v$, from which it follows that $|y-y_{Q'}|_v>\rho_v(f)$. Then $|y-\zeta|_v=|y-y_{Q'}|_v$ for any root $f(\zeta)=0$, and hence \[|f(y)|_v=|y-y_{Q'}|_v^d>\rho_v(f)^d.\] 
On the other hand, for any root $f(\zeta)=0$, we have \[|x-\zeta|_v\leq\max\{|x-x_{Q'}|_v, \rho_v(f)\}\leq |y-y_{Q'}|_v<|y-y_{Q'}|_v^d.\]
So we have
\[\left|y_{\phi(P)}-y_{Q'}\right|_v=\left|f(y)+x-y_{Q'}\right|_v=|y-y_{Q'}|^d_v>|y-x_{Q'}|_v=|x_{\phi(P)}-x_{Q'}|_v,\]
and so
\[\|\phi(P)-Q'\|_v=\|P-Q'\|_v^d.\]
We obtain by induction
\[\|\phi^N(P)-Q'\|_v=\|P-Q'\|_v^{d^N},\]
and so, in particular, $\|\phi^N(P)\|_v\to\infty$ as $N\to\infty$, whereupon $P\not\in \julia_{v, \phi}$.  If $|x-x_Q'|_v\geq |y-y_Q'|_v$, then a similar argument shows that  $\|\phi^{-N}(P)\|_v\to\infty$ as $N\to\infty$.  In either case, we cannot have $P\in \julia_{v, \phi}$.

Now, given that $\julia_{v, \phi}\subseteq \overline{D_v(Q'; \rho_v(f))}$, we will show the stronger containment \[\julia_{v, \phi}\subseteq\bigcup_{Q\in \Delta}\overline{D_v(Q; 1)}.\]  Suppose that $P=(x, y)\in \julia_{v, \phi}$, so that $\phi(P)\in \julia_{v, \phi}\subseteq \overline{D_v(Q'; \rho_v(f))}$.  It follows that 
\[|y-x_{Q'}|_v=|x_{\phi(P)}-x_{Q'}|_v\leq \rho_v(f),\]
 and 
 \[|x+f(y)-y_{Q'}|_v=|y_{\phi(P)}-y_{Q'}|_v\leq \rho_v(f).\]  Similarly, $\phi^{-1}(P)\in \julia_{v, \phi}\subseteq \overline{D_v(Q'; \rho_v(f))}$, and so $|x-y_{Q'}|_v\leq \rho_v(f)$ and $|y-f(x)-x_{Q'}|_v\leq \rho_v(f)$.  These combine to give 
 \begin{equation}\label{eq:f(x) f(y) small}|f(x)|_v, |f(y)|_v\leq \rho_v(f).\end{equation}
   Since there exists an $\eta$ with $f(\eta)=0$ and $|y-\eta|_v\geq \rho_v(f)$, we must have $|y-\zeta|_v\leq 1$ for some root $\zeta$ of $f$.  Similarly, there exists a root $\zeta'$ of $f$ with $|x-\zeta'|_v\leq 1$, and we have $P\in \overline{D_v((\zeta', \zeta); 1)}$.  Since $P$ was arbitrary, we have shown that $\julia_{v, \phi}\subseteq\bigcup_{Q\in \Delta_f}\overline{D_v(Q; 1)}$.
\end{proof}

The previous lemma tells us that every point $P\in\julia_{v, \phi}$ is distance at most one from a point of period dividing 2.  We see in the next lemma that the points in $K_{v, \phi}$ must cluster slightly more than this fact alone would indicate.  We remark, for the reader's convenience, that in the case that the set $X$ is infinite, in any of the statements below, the estimate $\#Y \geq \#X/N$ should be interpreted to mean that $Y$ is infinite as well.
\begin{lemma}\label{lem:benedetto main lemma}
Suppose that $X\subseteq\julia_{v, \phi}$, and that $\rho_v(f)>1$.  Then there is a subset $Y\subseteq X$ such that $\# Y\geq \#X/(3d^3)$, and such that for all $P_1, P_2\in Y$,
\[\max\{|y_{\phi^{-1}(P_1)}-y_{\phi^{-1}(P_2)}|, |y_{P_1}-y_{P_2}|, |y_{\phi(P_1)}-y_{\phi(P_2)}|\}\leq 1\]
and
\[\min\{|y_{\phi^{-1}(P_1)}-y_{\phi^{-1}(P_2)}|, |y_{P_1}-y_{P_2}|, |y_{\phi(P_1)}-y_{\phi(P_2)}|\}< 1.\]
\end{lemma}

\begin{proof}
Suppose that $P=(x, y)\in\julia_{v, \phi}$, and for convenience order the roots of $f(z)$, with multiplicity, as $\zeta_1, ..., \zeta_{d}$.  By equation~\eqref{eq:f(x) f(y) small} from the proof of Lemma~\ref{lem:benedetto filled julia}, we have $|f(y)|_v\leq \rho_v(f)$.  Since $|y-\zeta_i|_v\leq \rho_v(f)$ for all $i$, by Lemma~\ref{lem:benedetto filled julia}, and since there is some $i$ for which the inequality is sharp (otherwise all roots of $f(z)$ are contained in a disk of radius strictly less than $\rho_v(f)$), we may suppose, without loss of generality, that  $|y-\zeta_1|=\rho_v(f)$. We then have $\prod_{i\geq 2}|y-\zeta_i|_v\leq 1$, and so if it is not the case that $|y-\zeta_i|_v<1$, for some $i$, then we must have $|y-\zeta_i|_v=1$ for all $i\geq 2$. 
 
 Now, assuming that we are in the latter case, we similarly have \[\left|f(y_{\phi^{-1}(P)})\right|_v=|f(x)|_v\leq \rho_v(f),\]
 and so by the same reasoning, either $|y_{\phi^{-1}(P)}-\zeta_i|_v<1$ for some $i$, or else we have $|y_{\phi^{-1}(P)}-\zeta_j|_v=\rho_v(f)$ for some $j$, and $|x_P-\zeta_i|_v=|y_{\phi^{-1}(P)}-\zeta_i|_v=1$ for all $i\neq j$.  (In fact, we must have $j=1$, but we do not use this observation.)
But in this case we have,  for all $i\neq j$, 
\[|y_{\phi(P)}-\zeta_i|_v=|x+f(y)-\zeta_i|_v=\rho_v(f),\]
since $|f(y)|_v=\rho_v(f)>1=|x-\zeta_i|_v$. Then, as $|f(y_{\phi(P)})|_v\leq \rho_v(f)$ by \eqref{eq:f(x) f(y) small}, we have
\[|y_{\phi(P)}-\zeta_j|_v=\frac{|f(y_{\phi(P)})|_v}{\prod_{i\neq j}|y_{\phi(P)}-\zeta_i|_v}\leq \rho_v(f)^{2-d}<1,\]
where we use our assumption that $d\geq 3$.

To reiterate, we have shown that for every $P\in \julia_{v, \phi}$, at least one of $y_P$, $y_{\phi(P)}$, or $y_{\phi^{-1}(P)}$ is at distance strictly less than 1 from some root of $f(z)$, while Lemma~\ref{lem:benedetto filled julia} tells us that each is at distance at most 1 from some root of $f(z)$.  Now, to each $P\in X$, we associate the tuple \[(\epsilon, i, j, k)\in\{-1, 0,  1\}\times\{1, ..., d\}^3\] if and only if
\begin{equation}\label{eq:close to roots}
\left|y_{\phi^{-1}(P)}-\zeta_i\right|_v\leq 1,\ \left|y_{P}-\zeta_j\right|_v\leq 1, \text{ and }\left|y_{\phi(P)}-\zeta_k\right|_v\leq 1,
\end{equation}
\emph{and} the inequality involving $\phi^\epsilon(P)$ is strict.  It is possible that more than one tuple is associated to a given $P$, but what we have just shown is that every point is associated to at least one tuple.  There are $3d^3$ distinct tuples, and so the set $X$ must contain a subset $Y$ of size at least $\# X/(3d^3)$ consisting of points all associated to the same tuple.  For all $P_1, P_2\in Y$, we have
\begin{equation*}
\left|y_{\phi^{-1}(P_1)}-y_{\phi^{-1}(P_2)}\right|_v\leq 1,\ \left|y_{P_1}-y_{P_2}\right|_v\leq 1, \text{ and }\left|y_{\phi(P_1)}-y_{\phi(P_1)}\right|_v\leq 1,
\end{equation*}
by \eqref{eq:close to roots} and the ultrametric inequality, as well as
\[\left|y_{\phi^\epsilon(P_1)}-y_{\phi^\epsilon(P_2)}\right|_v<1.\]
This proves the lemma.
\end{proof}

Lemma~\ref{lem:benedetto main lemma} shows that the values $y_P$, for $P\in\julia_{v, \phi}$ cluster to a certain extent.  The main idea of the proof is to use this clustering to contradict the product formula for $F$, as in the proof of the main result of \cite{benedetto}.  It might be the case, though, that a given $X\subseteq\julia_{v, \phi}$ contains a large number of points on a given horizontal line $y=c$, in which case the clustering given by Lemma~\ref{lem:benedetto main lemma} is trivial.  Lemma~\ref{lem:benedetto constant y lemma} shows that in this case, the values $x_P$ cluster in a similarly useful way.

\begin{lemma}\label{lem:benedetto constant y lemma}
Suppose that $\rho_v(f)>1$, that $L\subseteq \affine^2(K_v)$ is a horizontal line, and that
 $X\subseteq \julia_{v, \phi}\cap L$.  Then there is a subset $Y\subseteq X$ with $\# Y\geq \# X/d$ and such that for all $P_1, P_2\in Y$ we have
\[|x_{P_1}-x_{P_2}|_v<1.\]
\end{lemma}

\begin{proof}
We first note that post-composition of a polynomial with a small translation does not change the value of $\rho_v$.  In particular, for any monic polynomial $g(z)\in\overline{F}[z]$ with $\rho_v(g)>1$, and any $c\in\overline{F}$ with $|c|_v\leq\rho_v(g)$, we have $\rho_v(g+c)=\rho_v(g)$.  To show this, we note first that $\rho_v$ is clearly unchanged by pre-composition with a translation, and so we may suppose that $g(0)=0$ and
\[\rho_v(g)=\max\{|\zeta|_v:g(\zeta)=0\}.\]
Now, let $i$ denote the number of roots $g(\zeta)=0$, counted with multiplicity, such that $|\zeta|_v=\rho_v(g)$.  We note that we must have $i<d$, since $g(0)=0$.  If we write
\[g(z)=z^d+m_{d-1}z^{d-1}+\cdots+m_1z,\]
then, by the ultrametric inequality, we have $|m_{d-i}|_v=\rho_v(g)^i$.
  Note that, since $i\neq d$, $m_{d-i}$ is also the coefficient of $z^{d-i}$ appearing in $g(z)+c$, and so $g(z)+c$ must have a root $\eta$ satisfying $|\eta|_v\geq \rho_v(f)$.  On the other hand, the constant term of $g(z)+c$ has size
  \[\prod_{f(\eta')=0}\left|\eta'\right|_v=|c|_v\leq \rho_v(g),\]
  and so $g(z)+c$ has a root $\eta'$ satisfying $|\eta'|_v\leq 1<\rho_v(g)$.  Noting that $|\eta-\eta'|_v=\rho_v(g)$, we've shown that $\rho_v(g+c)\geq \rho_v(g)$.  But we can now apply this argument to the post-composition of $g(z)+c$ with translation by $-c$ to obtain the opposite inequality.
  
Now suppose that $L$ is defined by $y=y_L$, for $y_L\in F$.  Then, assuming $X$ is non-empty, we have $|y_L-\zeta|_v\leq \rho_v(f)$ for all roots $\zeta$ of $f(z)$.  For any $P=(x, y_L)\in X$ we have, by Lemma~\ref{lem:benedetto filled julia},
\[|y_L-f(x)-\zeta|_v=|x_{\phi^{-1}(P)}-\zeta|_v\leq 1\]
for some root $\zeta$ of $f(z)$.  So, if $\eta_1, ..., \eta_d$ are the roots of $f(z)-y_L+\zeta$, listed with multiplicity, we must have
\[\prod_{i=1}^d\left|x-\eta_i\right|_v\leq 1.\]
But, since $\rho_v(f(z)-y_L+\zeta)=\rho_v(f)$, by the argument above, one of the terms in the product above must have size at least $\rho_v(f)>1$, and hence for some $i$ we have $|x-\eta_i|_v<1$.
The set $X$ must have a subset $Y$ of size at least $\# X/d$ such that for some $i$, $|x_P-\eta_i|_v<1$ for all $P\in Y$.  By the ultrametric inequality, this proves the lemma.
\end{proof}

The last piece needed for the proof of Theorem~\ref{th:benedetto} is the observation that, if $\phi$ is not isotrivial, then there is at least one place of $v$ with $\rho_v(f)>1$.

\begin{lemma}\label{lem:benedetto isotrivial}
Suppose that $\rho_v(f)=1$ for all $v\in M_F$.  Then there is an $\alpha\in \overline{F}$, such that $f(z+\alpha)\in \overline{K}[z]$.  In particular, $\phi(x, y)$ is affine-conjugate to a map defined over $\overline{K}$.
\end{lemma}

\begin{proof}
First, we note that $\rho_v(f(z))=\rho_v(f(z+\alpha))$ for any $\alpha\in \overline{F}$.  In particular, if $f(\zeta_1)=0$, we have $\rho_v(f(z+\zeta_1))=\rho_v(f(z))$.  But the roots of $f(z+\zeta_1)$ are precisely the elements of the form $\zeta_2-\zeta_1$, for $f(\zeta_2)=0$.  In particular, the roots $\eta$ of $f(z+\zeta_1)$ satisfy $|\eta|_v\leq 1$ for all $v\in M_F$.  By definition, this means that $\eta\in \overline{K}$ for every root $\eta$ of $f(z+\zeta_1)$.  In other words,
\[f(z+\zeta_1)=\prod_{\eta}(z-\eta)\in \overline{K}[z],\]
where the product is taken over roots of $f(z+\zeta_1)$ with the appropriate multiplicities.  We now apply Lemma~\ref{lem:affine conjugacy}.
\end{proof}

We can now complete the proof of Theorem~\ref{th:benedetto}.
\begin{proof}[Proof of Theorem~\ref{th:benedetto}]
If $\rho_v(f)= 1$ for all $v$, then by Lemma~\ref{lem:benedetto isotrivial} the map $\phi$ is isotrivial.  We shall assume, then, that this is not the case. 
Let $s\geq 1$ denote the number of places $v$ such that $\rho_v(f)>1$, and suppose that
\[X_0=\left\{P\in \affine^2(F):\hat{h}_\phi(P)=0\right\}\] contains at least $(3d^{6})^s+1$ elements.  Note that $X_0\subseteq \julia_{v, \phi}$ for each $v\in M_F$, and so for every place with $\rho_v(f)=1$, we have
\begin{equation}\label{eq:good reduction bound}|x_1-x_2|_v, |y_1-y_2|_v\leq 1\end{equation}
for any $(x_1, y_1), (x_2, y_2)\in X_0$, by Lemma~\ref{lem:benedetto filled julia}.
We suppose, at first, that there is a horizontal line $L\subseteq\affine^2(F)$  such that $Z=L\cap X_0$ contains at least $d^s+1$ elements.  Then, by $s$ applications of Lemma~\ref{lem:benedetto constant y lemma}, there exist at least two elements $(x_1, y_L), (x_2, y_L)\in Z$ such that
\[0<|x_1-x_2|_v<1\]
for every place $v$ with $\rho_v(f)>1$, the lower bound following from $(x_1, y_L)\neq (x_2, y_L)$.  Applying \eqref{eq:good reduction bound} at the remaining places,  we obtain
\[\prod_{v\in M_F}|x_1-x_2|_v<1,\]
an obvious contradiction to the product formula.  It must be the case, then, that any horizontal line in $\affine^2(F)$ meets $X_0$ in at most $d^s$ points.
 
Now, by $s$ applications of Lemma~\ref{lem:benedetto main lemma}, we may choose a subset $X_1\subseteq X_0$
such that
\begin{equation}\label{eq:product of three}
|y_{\phi^{-1}(P_1)}-y_{\phi^{-1}(P_2)}|_v\cdot|y_{P_1}-y_{P_2}|_v\cdot|y_{\phi(P_1)}-y_{\phi(P_2)}|_v< 1
\end{equation}
for all $P_1, P_2\in X_1$, and such that $\# X_1\geq X_0/(3d^3)^s>d^{3s}$.
Since at most $d^s$ of these points lie on any given horizontal line, we may choose a subset $X_2\subseteq X_1$ with $\# X_2\geq \# X_1/d^s$ such that $y_{P_1}\neq y_{P_2}$, for any distinct points $P_1, P_2\in X_2$.  Applying the same argument to $\phi(X_2)\subseteq X_0$, we may choose a subset $X_3\subseteq X_2$ such that $y_{\phi(P_1)}\neq y_{\phi(P_2)}$, for distinct $P_1, P_2\in X_3$, and such that $\# X_3\geq \# X_2/d^s>d^s$.  Finally, applying the same argument to $\phi^{-1}(X_3)\subseteq X_0$, we may choose an $X_4\subseteq X_3$ containing at least 2 distinct points $P_1, P_2$, such that $y_{P_1}\neq y_{P_2}$, $y_{\phi(P_1)}\neq y_{\phi(P_2)}$, and $y_{\phi^{-1}(P_1)}\neq y_{\phi^{-1}(P_2)}$.  For these two points we have \eqref{eq:product of three} at every place with $\rho_v(f)>1$, and \eqref{eq:good reduction bound} elsewhere, and so
\[
\prod_{v\in M_F}|y_{\phi^{-1}(P_1)}-y_{\phi^{-1}(P_2)}|_v\cdot|y_{P_1}-y_{P_2}|_v\cdot|y_{\phi(P_1)}-y_{\phi(P_2)}|_v<1.\]
But applying the product formula for the three non-zero elements $y_{P_1}-y_{P_2}$, $y_{\phi(P_1)}-y_{\phi(P_2)}$, and $y_{\phi^{-1}(P_1)}-y_{\phi^{-1}(P_2)}$ of $F$ contradicts this.  It follows that there were no more than $(3d^{6})^s$ points $P\in \affine^2(F)$ satisfying $\hat{h}_\phi(P)=0$.
\end{proof}

\section{Quadratic H\'{e}non maps, and the proof of Theorem~\ref{th:lang}}\label{sec:lang proof}

The proof of Theorem~\ref{th:lang} proceeds along similar lines to that of the main result of \cite{ads_lang}, which in turn is inspired by a result of Silverman \cite{silv_lang}.  The proof also bears resemblance to the proof of Theorem~\ref{th:benedetto}, relying on the same basic ideas.

  Throughout this section, $K$ will be either a number field, or a function field, $M_K$ will denote its set of places. We will denote the \emph{local degree} at $v\in M_K$ by $n_v$, where this is 1 if $K$ is a function field, and $n_v=\frac{[K_v:\QQ_v]}{[K:\QQ]}$ if $K$ is a number field. We assume that each valuation on $K$ has been extended in some way to the algebraic closure, and take \[\phi(x, y)=(y, x+y^2+b),\] for some $b\in K$.  Our first lemma is a slight improvement of Lemma~\ref{lem:local heights}, and follows from essentially the same argument.  Although the sharper bounds are not fundamentally needed in the proof of Theorem~\ref{th:lang}, they allow for numerically stronger results, and make the symmetry of this case somewhat more obvious.
We re-define, for this section only,
\[\basin^+_v(\phi)=\{(x, y)\in \affine^2(K):|y|_v^2>(3)_v\max\{|x|_v, |b|_v, 1\}\},\]
and
\[\basin^-_v(\phi)=\{(x, y)\in \affine^2(K):|x|_v^2>(3)_v\max\{|y|_v, |b|_v, 1\}\}.\]

\begin{lemma}\label{lem:local heights quadratic}
The set $\basin_v^{+}(\phi)$ is closed under the action of $\phi$, and
\[\hat{\lambda}^+_{v,\phi}(x, y)=\log|y|_v+\epsilon^+(b, P, v)\] for $(x, y)\in\basin^{+}_v(\phi)$,
where $\epsilon^+=0$ for $v\in M_K^{0}$, and $-\log 3\leq \epsilon^+\leq \log 5/3$ otherwise.
Similarly, the set $\basin_v^{-}(\phi)$ is closed under the action of $\phi^{-1}$; and
\[\hat{\lambda}^-_{v,\phi}(x, y)=\log|x|_v+\epsilon^-(b, P, v)\] for $(x, y)\in\basin^{-}_v(\phi)$,
where $\epsilon^-=0$ for $v\in M_K^{0}$, and $-\log 3\leq \epsilon^-\leq \log 5/3$ otherwise.
\end{lemma}

We will make use of the following simple result, which shows that points not in $\basin^+_v(\phi)\cup\basin^-_v(\phi)$ must cluster $v$-adically.  This result is similar in flavour to Lemma~\ref{lem:benedetto filled julia} above.
\begin{lemma}\label{lem:close to root}
Let $P=(x, y)\not\in\basin^+_v(\phi)\cup\basin^-_v(\phi)$.  Then \[\|x, y\|_v\leq \max(3)_v\{1, |b|_v\}^{1/2}.\]  If, in addition, $\phi^{-1}(P), \phi(P)\not\in\basin^+_v(\phi)\cup\basin^-_v(\phi)$, then there are roots $\gamma_1^2=-b$ and $\gamma_2^2=-b$ such that \[|x-\gamma_2|_v, |y-\gamma_1|_v\leq (12)_v|2|_v^{-1}.\]
\end{lemma}

Before proceeding with the proof of the lemma, we note that it follows from this that if $v\in M_K$ is non-archimedean, $v(b)<0$, $v(2)=0$, and $\phi$ has a  periodic point $(x, y)\in\affine^2(K)$, then $v(x)=v(y)=\frac{1}{2}v(b)$.  From this we conclude that $v(b)$ must be even, a fact which simplifies our calculations in Section~\ref{sec:computations} (see Lemma~\ref{lem:heights used for computations}).

\begin{proof}[Proof of Lemma~\ref{lem:close to root}]
We treat the non-archimedean case first.  Suppose that $|b|>1$, and that $|y|=((3)_v|b|^{1/2})^c$, for some $c>1$.  Then
\[(3)_v^{2c}|b|^{c}=|y|^2\leq (3)_v\max\{|b|, |x|\},\]
and so $|x|\geq (3)_v^{2c-1}|b|^{c}>1$.  But then it is also the case that
\[(3)_v^{4c-2}|b|^{2c}\leq |x|^2\leq (3)_v\max\{|b|, |y|\}=(3)_v|y|=(3)_v^{1+c}|b|^{c/2},\]
and so \[|b|^{3c/2}\leq (3)_v^{3c-3}|b|^{3c/2}\leq 1.\]  This contradicts $|b|>1$ and $c>1$.  The proof that $|x|\leq (3)_v|b|^{1/2}$ is identical.  

If, on the other hand, $|b|\leq 1$, then the inequalities $|y|^2\leq (3)_v\max\{1, |x|\}$ and $|x|^2\leq (3)_v\max\{1, |y|\}$ immediately imply $|x|, |y|\leq (3)_v$.

For the proof of the second part of the lemma.
Supposing that $\phi^{-1}(P), \phi(P)\not\in\basin^+_v(\phi)\cup\basin^-_v(\phi)$, and $|b|>1$,  we have
\[|y^2+b|\leq (2)_v\max\{|x|,|x+y^2+b|\}\leq (6)_v|b|^{1/2}.\]
Letting $\gamma_1^2=-b$, chosen without loss of generality so that $|y-\gamma_1|\leq |y+\gamma_1|$, we have
\[|2\gamma_1|=|(\gamma_1-y)+(\gamma_1+y)|\leq (2)_v|y+\gamma_1|,\]
and so
\[|y-\gamma_1|\leq \frac{(6)_v|b|^{1/2}}{|y+\gamma_1|}\leq \frac{(12)_v|b|^{1/2}}{|2\gamma_1|}=(12)_v|2|_v^{-1}.\]
A similar argument gives $|x-\gamma_2|\leq (12)_v|2|_v^{-1}$.

If $|b|\leq 1$, then the claim follows directly from the fact that
\[|y-\gamma_1|\leq (2)_v\max\{|y|, |\gamma_1|\}\leq (6)_v,\]
and similarly for $x$.
\end{proof}

We now come to the four main technical lemmas used in the proof of Theorem~\ref{th:lang}.  Before stating the lemmas, we introduce some 
 useful notation.  If $M, N\in \ZZ$, let
\[[M, N]=\{M, M+1, ..., N-1, N\},\]
and for $I\subseteq [M, N]$,
\[\conv(I)=[\min(I), \max(I)].\]
We will also use $P_j=(x_j, y_j)$ to denote $\phi^j(P)$.

\begin{lemma}\label{lem:two-sided}
Let $I\subseteq[-M, M]$ such that $\#I\geq 2$, and suppose that $v$ is archimedean, or non-archimedean with $|b|_v>1$.  Then there exists a subset $J\subseteq I$ with $\#J \geq \frac{1}{18}\# I-1$ such that for all $i\neq j\in J$,
\begin{equation}\label{eq:two-sided}\log|x_i-x_j|+\log|y_i-y_j|+\lambda(b)\leq 3\cdot 2^{M-1}\hat{\lambda}_{\phi, v}(P)+\alpha_v,\end{equation}
where by convention the inequality holds if $x_i=x_j$ or $y_i=y_j$, and where
\[
\alpha_v=\begin{cases}
18 & \text{if $v$ is archimedean}\\
6\log|2|_v^{-1} & \text{otherwise}.
\end{cases}
\]
\end{lemma}

\begin{proof}
We first suppose that there is a subset $J_0\subseteq I$ such that $\# J_0\geq \frac{1}{18}\# I$, and such that $P_j\in\basin^+_v(\phi)$ for all $j\in J_0$.  Then for all $j\in J_0$, we have
\[
\log|y_j|\leq \hat{\lambda}_{v, \phi}^+(P_j)+(\log 3)_v\leq 2^M\hat{\lambda}_{v, \phi}^+(P)+(\log 3)_v.
\]
At the same time, as long as $j\neq\min(J_0)$, we have $P_{j-1}\in\basin_v^+(\phi)$, and so 
\[\log|x_j|=\log|y_{j-1}|\leq 2^{M-1}\lambda^{+}_{\phi, v}(P)+(\log 3)_v.\]
It is also the case that if $i\neq j$
\[\lambda(b)< 2\log\min\{|y_i|, |y_j|\}-(\log 3)_v\leq 2^M\lambda^{+}_{\phi, v}(P)+(\log 3)_v.\]
So, if we take $J=J_0\setminus\{\min(J_0)\}$, we have
\begin{eqnarray*}
\log|x_i-x_j|+\log|y_i-y_j|+\lambda(b)&\leq& \log\max\{|x_i|,|x_j|\}+\log\max\{|y_i|,|y_j|\}\\&&+\lambda(b)+(\log 4)_v\\
&\leq &2^M\hat{\lambda}_{v, \phi}^+(P)+2^{M-1}\lambda^{+}_{\phi, v}(P)\\&&+2^M\lambda^{+}_{\phi, v}(P)+(\log 108)_v\\
&\leq &3\cdot 2^{M}\hat{\lambda}_{\phi, v}(P)+(\log 108)_v,
\end{eqnarray*}
for all $i\neq j\in J$.  We also note that $\# J\geq \frac{1}{18}\#I-1$.

A similar argument shows the required inequality in the case that there is a subset $J_0\subseteq I$ such that $\# J_0\geq \frac{1}{18}\# I$ and such that $P_j\in \basin_v^-(\phi)$ for all $i\in J_0$.  So we will assume that no such set exists.  It follows that there is a subset $J_0\subseteq I$ with $\#J_0\geq \frac{8}{9}\#I$, such that $P_j\not\in\basin^{+}_v(\phi)\cup\basin^{-}_v(\phi)$ for all $j\in J_0$.
Then for all but at most four elements $j\in \conv(J_0)$, we have \[P_{j-2}, P_{j-1}, P_j, P_{j+1}, P_{j+2}\in\affine^2(K)\setminus( \basin^{+}_v(\phi)\cup\basin^{-}_v(\phi)).\]  By Lemma~\ref{lem:close to root}, we may choose for each $j\in [\min(J_0)+2, \max(J_0)-2]$ roots \[\gamma_{1, j}^2=\gamma_{2, j}^2=\gamma_{3, j}^2=\gamma_{4, j}^2=-b\] such that
\[|x_j-\gamma_{1, j}|, |x_{j-1}-\gamma_{2, j}|, |y_j-\gamma_{3, j}|, |y_{j+1}-\gamma_{4, j}|\leq (12)_v|2|_v^{-1}.\]
By the pigeonhole principle, there is a subset $J\subseteq J_0$ with \[\# J\geq  \frac{1}{16}(\# J_0-4)\geq \frac{1}{18}\# I-1\] such that $\gamma_{n, j}$ is the same for all $j\in J$, for each $n$.  It follows that
\[|x_i-x_j|, |y_i-y_j|, |x_{i-1}-x_{j-1}|, |y_{i+1}-y_{j+1}|\leq (24)_v|2|_v^{-1}\]
for all $i, j\in J$.  Now,  for $i, j\in J$,
\begin{eqnarray*}
\left|y_i^2-y_j^2\right|&\leq& (2)_v\max\left\{|y_i^2-y_j^2+x_i-x_j|, |x_i-x_j|\right\}\\
&=&(2)_v\max\left\{|y_{i+1}-y_{j+1}|, |x_i-x_j|\right\}\\&\leq& (48)_v|2|_v^{-1}.
\end{eqnarray*}
Now, if $v$ is non-archimedean, and $|2|_v=1$, then
\[|y_i+y_j|=\max\{|2y_j|, |y_i-y_j|\}=|b|_v^{1/2},\] which means that $|y_i-y_j|\leq |b|^{-1/2}$.  Similarly 
\[|x_i^2-x_j^2+y_j-y_i|=|x_{i-1}-x_{j-1}|\leq 1,\]  which gives $|x_i-x_j|\leq |b|^{-1/2}$.
Combining these gives
\[\log|x_i-x_j|+\log|y_i-y_j|+\log|b|\leq 0,\]
for all $i, j\in J$.

If $v$ is non-archimedean, but $|2|_v<1$, then we consider two cases.  If $|b|_v>|2|_v^{-4}$, then
$|y_j-\gamma_{3, j}|\leq |2|_v^{-1}$ gives $|y_j|=|\gamma_{3, j}|=|b|^{1/2}>|2|_v^{-2}$.  By the argument above, this then gives $|y_j+y_i|=|2y_j|=|2|_v|b|^{1/2}$, and so
\[|y_j-y_i|\leq |2|_v^{-2}|b|^{-1/2}.\]
If, on the other hand, $|b|_v\leq |2|_v^{-4}$, then we have at once
\[|y_j-y_i|\leq |2|_v^{-1}\leq |2|_v^{-3}|b|^{-1/2}.\]
Obtaining the same estimates for $|x_j-x_i|$, we have
\[\log|x_j-x_i|+\log|y_j-y_i|+\log|b|\leq 6\log|2|_v^{-1}\]
for all $i, j\in J$.

Finally, if $v$ is archimedean, we again have two cases.  If $|b|_v> 24^2$, then we have
\[|b|^{1/2}=|\gamma_{3, j}|\leq |y_j| + |y_j-\gamma_{3, j}|\leq |y_j|+ 6\leq |y_j|+\frac{1}{4}|b|^{1/2},\]
and so $|y_j|\geq \frac{3}{4}|b|^{1/2}$.  It follows that
\[|y_j+y_i|\geq |2y_j|-|y_j-y_i|\geq \frac{3}{4}|2b|^{1/2}-12\geq \frac{1}{4}|b|^{1/2}.\]
From this we obtain
\[|y_j-y_i|\leq \frac{24}{|y_j+y_i|}\leq 96|b|^{-1/2}.\]
If, on the other hand, $|b|_v\leq 24^2$, then
\[|y_j-y_i|\leq 12\leq  6912 |b|^{-1/2}.\]
Obtaining the same estimates for $|x_j-x_i|$, we have
\[\log|x_j-x_i|+\log|y_j-y_i|+\log|b|\leq 2\log 6912\]
for all $i, j\in J$.

In each case, the estimate \eqref{eq:two-sided} now follows from the fact that $\hat{\lambda}_\phi(P)$ is non-negative.
\end{proof}

\begin{lemma}\label{lem:constant y}
Let $I\subseteq[-M, M]$ such that $\#I\geq 2$, and $y_i=y_j$  for all $i, j\in I$.  Then there exists a subset $J\subseteq I$ with $\#J \geq \frac{1}{5}\# I-1$ such that for all $i, j\in J$,
\begin{equation}\label{eq:constant y}\log|x_i-x_j|+\frac{1}{2}\lambda(b)\leq 2^{M+1}\hat{\lambda}_{\phi, v}^-(P)+\beta_v,\end{equation}
where by convention the inequality holds if $x_i=x_j$, and
\[\beta_v=\begin{cases}
8 & \text{if $v$ is archimedean}\\
2\log|2|_v^{-1}& otherwise.
\end{cases}\]
\end{lemma}

\begin{proof}
First, suppose that there is a subset $J\subseteq I$ such that $\# J\geq \frac{1}{5}\# I$, and $P_j\in \basin^-_v(\phi)$ for all $j\in J$.  Then for any $i, j\in J$, we have
\begin{eqnarray*}
\log|x_i-x_j|+\frac{1}{2}\lambda_v(b)&\leq &\log\max\{|x_i|, |x_j|\}+\frac{1}{2}\lambda_v(b)+\log(2)_v\\
&\leq& 2\log\max\{|x_i|, |x_j|\}+\left(\log2-\frac{1}{2}\log 3\right)_v\\
&=&2\max\{2^i\hat{\lambda}_{\phi, v}^-(P), 2^j\hat{\lambda}_{\phi, v}^-(P)\}+\left(\frac{1}{2}\log 4/3\right)_v\\
&\leq & 2^{M+1}\hat{\lambda}^-_{\phi, v}(P)+\left(\frac{1}{2}\log 4/3\right)_v.
\end{eqnarray*}
In this case we are done, and so we will assume from this point forward that such a $J\subseteq I$ does not exist.

Now, we note that since $y_i=y_j$ for all $i, j\in I$, it must be the case that $P_i\not\in \basin^{+}_v(\phi)$, except perhaps for $i=\max(I)$, or if $v$ is archimedean and $|b|_v\leq 75$.  To see that this is true, note that if $(x_{i_1}, y), (x_{i_2}, y)\in \basin^{+}_v(\phi)$ with $i_2>i_1$, and $v$ is non-archimedean, then by Lemma~\ref{lem:local heights quadratic} we have
\begin{eqnarray*}
2^{i_2}\hat{\lambda}_{\phi, v}^+(P)&=&\hat{\lambda}_{\phi, v}^+(P_{i_2})\\
&=& \log|y|\\
&=&\hat{\lambda}_{\phi, v}^+(P_{i_1})\\
&=&2^{i_1}\hat{\lambda}_{\phi, v}^+(P).
\end{eqnarray*}
It follows immediately that $i_1=i_2$.  If $v$ is archimedean, and $i_2\geq i_1+1$, we have
\begin{eqnarray*}
\log|y|&\leq& \hat{\lambda}_{\phi, v}^+(P_{i_1})+\log 3\\&=&2^{i_1-i_2}\hat{\lambda}_{\phi, v}^+(P_{i_2})+\log 3\\
&\leq &2^{i_1-i_2}(\log|y|+\log5/3) +\log 3\\
&\leq & \frac{1}{2}\log|y| + \frac{1}{2}\log 15.
\end{eqnarray*}
It follows that
\[\frac{1}{2}\log^+(b)+\frac{1}{2}\log 3\leq \log |y|\leq \log 15,\]
and consequently $|b|\leq 75$.  In this case, we can choose a set $J\subseteq I$ with $\#J \geq \frac{4}{5}\#I$ and $P_j\not\in\basin^{-}_v(\phi)$ for all $j\in J$.  For these $j$, though, we then have
\[2\log |x_j|\leq \log\max\{1, |b|, |y|\}+\log 3\leq \log 225,\]
and so
\[\log |x_j-x_i|+\log^+|b|\leq \log\max\{|x_j|, |x_i|\}+\log^+|b|+\log 2\leq \log 2250\]
for all $i, j\in J$.  Since this verifies the claim, we will henceforth suppose that $P_i\not\in\basin^+_v(\phi)$, except possibly for $i=\max(I)$.

We have assumed that there is a subset $J_0\subseteq I$ with $\#J_0 >\frac{4}{5}\#I-1$, and \[P_j\in\affine^2(K)\setminus(\basin^+_v(\phi)\cup\basin^-_v(\phi))\] for all $j\in J_0$.
It follows that for all but at most two values $j\in \conv(J_0)$, we have \[P_{j-2}, P_{j-1}, P_{j}\in\affine^2(K)\setminus( \basin^+_v(\phi)\cup\basin^-_v(\phi)).\]
From Lemma~\ref{lem:close to root}, we see that for each such $j$,
\[|x_{j-1}-\gamma_j|\leq (12)_v|2|_v^{-1}\] for some root $\gamma_j^2=-b$.  We may then choose a subset $J_1\subseteq J_0$, with \[\# J_1\geq \frac{1}{2}(\# J_0-2)>\frac{2}{5}\# I-\frac{3}{2}\]  such that $|x_{j-1}-\gamma|\leq (12)_v|2|_v^{-1}$ for all $j\in J_1$, for one particular $\gamma^2=-b$ which does not depend on $j$.  Now, for $j\in J_1$, we have
\[ |x_{j}^2-(y-b+\gamma)|=|x_{j-1}-\gamma|\leq (12)_v|2|_v^{-1}.\]

First we treat the case in which $v$ is non-archimedean.
If $\delta^2=y-b+\gamma$, then $|\delta|=|b|^{1/2}$, and so
\[|2|_v|b|_v^{1/2}=|2\delta|_v=|\delta-x_j+\delta+x_j|_v\leq \max\{|x_j\pm \delta|_v\},\]
and so
\[\min\{|x_{j}\pm \delta|_v\}\leq |2|^{-2}_v|b|_v^{-1/2}.\]
We may now choose a subset $J\subseteq J_1$ with $\# J\geq \frac{1}{2}\# J_0$, such that \[|x_i-x_j|\leq |2|_v^{-2}|b|^{-1/2}\] for all $i, j\in J$.  It follows that for all such $i, j$,
\[\log|x_i-x_j|+\frac{1}{2}\lambda(b)\leq 0\leq 2^{M+1}\hat{\lambda}_{\phi, v}^-(P)+2\log|2|_v^{-1}.\]
We note that $\#J > \frac{1}{5}\# I-\frac{3}{4}$.

We proceed similarly if $v$ is archimedean.
If $\delta^2=y-b+\gamma$, then we have $|\gamma|=|b|^{1/2}$ and $|y|\leq 3|b|^{1/2}$ by Lemma~\ref{lem:close to root}.  Since we may suppose that $|b|>76$, we have
\[|y-b+\gamma|\geq |b|-4|b|^{1/2}\geq \frac{1}{2}|b|,\]
and so $|\delta|\geq \frac{1}{\sqrt{2}}|b|^{1/2}$.
We then have
\[\sqrt{2}|b|_v^{1/2}\leq |2\delta|_v=|\delta-x_j+\delta+x_j|_v\leq 2\max\{|x_j\pm \delta|_v\},\]
and so
\[\min\{|x_{j}\pm \delta|_v\}\leq\frac{6}{\max\{|x_j\pm \delta|_v\}}\leq 6\sqrt{2}|b|^{-1/2}.\]
We may now choose a subset $J\subseteq J_1$ with $\# J\geq \frac{1}{2}\# J_0$, such that \[|x_i-x_j|\leq 12\sqrt{2}|b|^{-1/2}\] for all $i, j\in J$.  It follows that for all such $i, j$,
\[\log|x_i-x_j|+\frac{1}{2}\lambda(b)\leq 0\leq 2^{M+1}\hat{\lambda}_{\phi, v}^-(P)+\frac{1}{2}\log 288.\]
We note again that $\#J > \frac{1}{5}\# I-\frac{3}{4}$.
\end{proof}

The proof of the following lemma is a straight-forward modification of the proof of Lemma~\ref{lem:constant y}.

\begin{lemma}\label{lem:constant x}
Let $I\subseteq[-M, M]$ such that $\#I\geq 2$,  and $x_i= x_j$ for all $i, j\in I$.  Then there exists a subset $J\subseteq I$ with $\#J \geq \frac{1}{5}\# I-1$ such that for all $i, j\in J$,
\begin{equation}\label{eq:constant x}\log|y_i-y_j|+\frac{1}{2}\lambda(b)\leq 2^{M+1}\hat{\lambda}_{\phi, v}^+(P)+\beta_v,\end{equation}
where by convention the inequality holds if $x_i=x_j$.
\end{lemma}

The three lemmas above treat the case of $v$ archimedean, or a place of bad reduction.  The final lemma treats the good reduction primes.
\begin{lemma}\label{lem:good reduction}
Suppose that $|b|\leq 1$ and that $v$ is non-archimedean.  Then for any $i, j\in [-M, M]$, we have
\[\log|x_i-x_j|\leq 2^{M+1}\hat{\lambda}_{\phi, v}(P)\]
and
\[\log|y_i-y_j|\leq 2^{M+1}\hat{\lambda}_{\phi, v}(P).\]
\end{lemma}

\begin{proof}
If $P_i=(x, y)\in \basin^-_v(\phi)$, then we have
\[\log|x|=\hat{\lambda}_{v, \phi}^-(P_i)\leq  2^{M+1}\hat{\lambda}_{\phi, v}^-(P).\]
If $P_i\not\in\basin^-_v(\phi)$, then we have $|x|^2\leq \max\{1, |y|\}$.  If $|y|\leq 1$, then $|x|\leq 1$, and so we have
\[\log|x|\leq 0\leq 2^{M+1}\hat{\lambda}_{\phi, v}^-(P).\]
If, on the other hand, $|y|>1$, then we have $|y|^2>\max\{1, |x|\}$, and so $P\in \basin^+_v(\phi)$.  In this case,
\[\log|x|\leq 2\log|y|=2\hat{\lambda}_{v, \phi}^+(P_i)\leq 2^{M+1}\hat{\lambda}_{v, \phi}^+(P).\]
In any case,
\[\log|x_i-x_j|\leq \log\max\{|x_i|, |x_j|\}\leq 2^{M+1}\left(\hat{\lambda}_{\phi, v}^-(P)+\hat{\lambda}_{\phi, v}^+(P)\right).\]
The second inequality is similar.
\end{proof}

To begin, we note that by Lemma~\ref{lem:local heights quadratic}, the canonical heights defined by Kawaguchi \cite{kawaguchi} may be written as 
\[\hat{h}_\phi^+(P)=\sum_{v\in M_K}\frac{[K_v:\QQ_v]}{[K:\QQ]}\hat{\lambda}_{v, \phi}^+(P)\]
and
\[\hat{h}_\phi^-(P)=\sum_{v\in M_K}\frac{[K_v:\QQ_v]}{[K:\QQ]}\hat{\lambda}_{v, \phi}^-(P).\]
We define an array of rational numbers as follows.  Let $B_{0, 0}=2$, let $B_{0, n+1}=5B_{0, n}+\frac{5}{2}$, and let $B_{m+1, n}=18B_{m, n}+18$.  Now, fix $s$, suppose that $b\in K$ is $s$-integral, and choose $M\in\ZZ^+$ such that $2M\geq B_{s, s}$.  Fix $P\in \affine^2(K)$, and suppose that $P$ is not periodic of period less than $2M$.  In other words, suppose that the points $P_i$ are distinct, for $i\in [-M, M]$.  Applying Lemma~\ref{lem:two-sided} to each of the (at most) $s$ places of bad reduction, we may choose a subset $I\subseteq [-M, M]$ with $\# I\geq B_{0, s}$ such that for all $i, j\in I$ and all places $v\in M_K$, we have
\begin{equation}\label{eq:two-sided thm}\log|x_i-x_j|+\log|y_i-y_j|+\lambda(b)\leq 2^{M+2}\left(\hat{\lambda}_{\phi, v}^-(P)+\hat{\lambda}_{\phi, v}^+(P)\right)+\alpha_v,\end{equation}
where the relation follows from Lemma~\ref{lem:good reduction} for the places of good reduction.
Suppose that there exist two values $i, j\in I$ such that $x_i\neq x_j$ and $y_i\neq y_j$.  Then summing \eqref{eq:two-sided thm} over all places, with appropriate weights, gives
\[h(b)\leq 2^{M+2}\hat{h}_\phi(P)+C,\]
for some constant $C\leq 23$.  In this case, the inequality in Theorem~\ref{th:lang} follows for all but finitely many $b\in K$.  For the rest, we use the fact that $\hat{h}_\phi$ is discrete, which follows from the results of Kawaguchi \cite{kawaguchi}.

Now consider the case that there do not exist values $i, j\in I$ with $x_i\neq x_j$ and $y_i\neq y_j$.  Then we either have, for all $i, j\in I$ $x_i=x_j$, or else for all $i\in I$ $y_i=y_j$.  In the former case, we may apply Lemma~\ref{lem:constant x} to choose a subset $J\subseteq I$ with $\#J \geq 2$, and 
\begin{equation}\label{eq:constant x thm}\log|y_i-y_j|+\frac{1}{2}\lambda(b)\leq 2^{M+1}\hat{\lambda}_{\phi, v}^+(P)+\beta_v\end{equation}
for all $i, j\in J$ and all $v\in M_K$ (the relation holds for places of good reduction by Lemma~\ref{lem:good reduction}).  Note that, since $x_i=x_j$ for all $i, j\in I$, we have $y_i\neq y_j$.  Choosing $i\neq j\in J$, and summing \eqref{eq:constant x thm} with the appropriate weights, we obtain
\[h(b)\leq 2^{M+2}\hat{h}_\phi^+(P)+C\leq 2^{M+2}\hat{h}_\phi(P)+C,\]
for some constant $C\leq 10$.  Theorem~\ref{th:lang} follows from this in the case that $y_i=y_j$ for all $i, j\in I$, and the case where $y_i=y_j$ for all $i, j\in I$ is similar.  This proves Theorem~\ref{th:lang}.

\section{Variation in families, and  the proof of Theorem~\ref{th:tate}}\label{sec:tate proof}

For this section, we fix a number field or function field $K$, and a smooth, projective curve $C/K$, and let $F=K(C)$. Throughout, we will denote the \emph{local degree} at $v\in M_K$ by $n_v$, where this is 1 if $K$ is a function field, and $n_v=\frac{[K_v:\QQ_v]}{[K:\QQ]}$ if $K$ is a number field. For the benefit of the reader we will recall the germane properties of local height functions, based on the exposition of Lang \cite{lang}.  By an \emph{$M_K$-divisor}, we mean a function $\mathfrak{e}:M_K\to\RR$ such that $\mathfrak{e}(v)=1$ for all but finitely many places $v$.  For any effective divisor $D\in \Div(C)\otimes\QQ$, a set of local heights for $D$ will be a collection of functions $\lambda_{v, D}:C(\overline{K}_v)\to \RR$ such that for any choice of functions $w_\beta\in F$, with $w_\beta$ vanishing only at $\beta$, there exist $M_K$-divisors $\mathfrak{e}_{\beta}$ and $\mathfrak{d}$ such that 
\[\left|\lambda_{v, D}(t)\right|_v\leq \log\mathfrak{d}(v)\]
if $|w_\beta(t)|_v\geq \mathfrak{e}_{\beta}(v)$ for all $\beta\in\Supp(D)$, and
\[\left|\lambda_{v, D}(t)+m_\beta\frac{\log|w_\beta(t)|_v}{\ord_\beta(w_\beta)}\right|\leq \log\mathfrak{d}_v\]
if $|w_\beta(t)|_v<\mathfrak{e}_{v, \beta}$, where $m_\beta$ is the weight of $(\beta)$ in $D$.  For any Galois extension $L/K$, one defines
\[h_D(t)=\frac{1}{\Gal(L/K)}\sum_{\sigma\in\Gal(L/K)}\sum_{v\in M_K}n_v\lambda_{v, D}(t^\sigma),\]
for $t\in C(L)$, and it is easy to check that this gives a well-defined function $h_D:C(\overline{K})\to \RR$.
Although this definition depends on the choice of local heights, it is easy to show that a different choice of local heights changes the function by only a bounded amount.

For convenience, we will identify the points of $C$ over $\overline{K}$ with the set of places of $F$.
Now, if the H\'{e}non map $\phi(x, y)=(ay, x+f(y))$ has coefficients in $F$, we set
\[D_+(\phi, P)=\sum_{\beta\in C(\overline{K})}\hat{\lambda}^+_{\beta, \phi}(P)(\beta)\]
and
\[D_-(\phi, P)=\sum_{\beta\in C(\overline{K})}\hat{\lambda}^-_{\beta, \phi}(P)(\beta).\]
Note that, \emph{a priori}, we have $D_{\pm}(\phi, P)\in \Div(C)\otimes\RR$.  But it follows from Lemma~\ref{lem:local heights} that for each $\beta$, either $\hat{\lambda}_{\beta, \phi}^+(P)=0$ or else there is an $N\geq 1$ such that $\phi^N(P)\in\basin^+_\beta(\phi)$, and so
\[d^N\hat{\lambda}^+_{\beta, \phi}(P)=\hat{\lambda}^+_{\beta, \phi}(\phi^N(P))=\log|y_{\phi^N(P)}|_\beta\in\ZZ.\] It follows that $D_{+}(\phi, P)\in \Div(C)\otimes\QQ$, and similarly for $D_-(\phi, P)$.

We extend the constant field so that $\Supp(D)\subseteq C(K)$.
To each point $\beta\in \Supp(D)$ we associate a function $w_\beta\in K(C)$ which vanishes at $\beta$, and nowhere else, and we define a distance function by
\[\delta_v(\beta, t)=|w_\beta(t)|_v^{1/\ord_\beta(w_\beta)}.\]
Note that this function depends on the choice of $w_\beta$, although a different choice of $w_\beta$ only changes $\delta_v(\beta, \cdot)$ by a non-zero constant multiple as $t$ approaches $\beta$.  We also choose a system of local heights $\lambda_{v, D}$ as above, and note that
\[\left|\lambda_{v, D}(t)+m_\beta\log\delta_v(\beta, t)\right|\leq \log\mathfrak{d}(v)\]
if $\delta_v(\beta, t)<\mathfrak{e}(v)$ for some $\beta$, and $|\lambda_{v, D}(t)|_v\leq \log\mathfrak{d}(v)$ otherwise.

The following lemma is a simple consequence of the arguments in the proof of Lemma~11 of \cite{variation}.
\begin{lemma}\label{lem:orders}
For any function $g_1\in K(C)$ with a pole at $\beta$, and any $M_K$-divisor $\mathfrak{d}$, we can choose an $M_K$-divisor $\mathfrak{e}$ such that
\[|g_1(t)|_v>\mathfrak{e}_v\]
whenever $\delta_v(\beta, t)<\mathfrak{d}_v$.
For any function $g_2\in K(C)$ with neither a pole nor a zero at $\beta$, we can choose an $M_K$-divisor $\mathfrak{e}$ such that
\[\mathfrak{e}_v^{-1}\leq |g_2(t)|_v\leq \mathfrak{e}_v\]
whenever $\delta_v(\beta, t)<\mathfrak{d}_v$.
\end{lemma}

\begin{lemma}\label{lem:tate lemma main}
Let $\phi$, $P$, and $D_+$ be as above, and suppose that  $\beta\in\Supp(D_+)$ satisfies $P\in\basin_\beta^+(\phi)$.  Then there exist $M_K$-divisors $\mathfrak{e}$ and $\mathfrak{d}$ such that
\[\left|\hat{\lambda}_{v, \phi_t}^+(P_t)-\lambda_{v, D_+}(t)\right|\leq \mathfrak{e}_v\]
for all $t\in C(\overline{K}_v)$ with $0<\delta_v(\beta, t)< \mathfrak{d}_v$.
\end{lemma}

\begin{proof}
Since $P=(x,y)\in\basin_\beta^+(\phi)$, we have that $y^d/x\in K(C)$ has a pole at $\beta$, and so by Lemma~\ref{lem:orders} we can choose $\mathfrak{d}_0$ small enough that
$|y_t^d/x_t|_v>(d+2)_v^d$
whenever $\delta_v(\beta, t)<\mathfrak{d}_0(v)$.  Proceeding similarly for the functions $y^{d-i}/b_i$ and $y^{d-1}/a$, and noting that a pointwise minimum of $M_K$-divisors is again an $M_K$-divisor, we can construct a $\mathfrak{d}_1$ such that $P_t\in\basin^+_v(\phi_t)$ for all $t\in C(\overline{K}_v)$ satisfying $0<\delta_v(\beta, t)<\mathfrak{d}_1(v)$.  Now, for these $t$, we have 
\[\hat{\lambda}^+_{v, \phi_t}(P_t)=\log|y_t|_v+\epsilon_v(t),\]
where $\epsilon_v(t)=0$ for $v\in M_K$ non-archimedean, and $\epsilon_v(t)$ bounded in terms of $d$ otherwise, by Lemma~\ref{lem:local heights}.  On the other hand, since 
\[\hat{\lambda}^+_{\beta, \phi}(P)=\log|y|_\beta=-\ord_\beta(y),\] it follows that 
$-\ord_\beta(y)$ is the weight to which the prime divisor $(\beta)$ occurs in $D_+$.  In other words, by the defining properties of local heights, there exist $M_K$-divisors $\mathfrak{e}_0$ and $\mathfrak{d}_2$ such that
\[\left|\lambda_{v, D_+}(t)-\ord_\beta(y)\log\delta_v(\beta, t)\right|\leq \log\mathfrak{e}_0(v)\]
whenever $0<\delta_v(\beta, t)<\mathfrak{d}_2(v)$.
Finally, applying Lemma~\ref{lem:orders} to the function $g_2=y^{-\ord_{\beta(w_\beta)}}/w_\beta^{\ord_\beta(y)}$, 
we obtain $M_K$-divisors $\mathfrak{e}_1$ and $\mathfrak{d}_3$ such that
\[\big|\ord_\beta(w_\beta)\log |y_t|_v+\ord_\beta(y)\log|w_\beta(t)|_v\big|\leq \log\mathfrak{e}_1(v)\]
whenever $0<\delta_v(\beta, t)<\mathfrak{d}_3(v)$.

Combining these three estimates, we obtain (for $r_d$ determined by Lemma~\ref{lem:local heights})
\begin{equation}
\left|\hat{\lambda}_{v, \phi}(P_t)-\lambda_{v, D_+}(t)\right|\leq \log \left(\mathfrak{e}_0(v)\mathfrak{e}_1(v)\left(r_d\right)_v\right)
\end{equation}
so long as \[0<\delta_v(\beta, t)<\min\{\mathfrak{d}_1(v), \mathfrak{d}_2(v), \mathfrak{d}_3(v)\}.\]  Since pointwise minima and products of $M_K$-divisors are again $M_K$-divisors, this proves the lemma.
\end{proof}

\begin{lemma}\label{lem:tate lemma bad case}
Let $\phi$, $P$, and $D^+$ be as above, and suppose that $\beta\not\in\Supp(D^+)$ is a point at which some $b_i$, or $a$, or $x$ or $y$ has a pole, but such that $\hat{\lambda}_{\beta, \phi}(P)=0$.  Then there exist $M_K$-divisors $\mathfrak{e}$ and $\mathfrak{d}$ such that
\[\max\left\{\hat{\lambda}_{v, \phi_t}^+(P_t),\lambda_{v, D}(t)\right\}\leq\mathfrak{e}_v\]
for all $t\in C(\overline{K_v})$ with $\delta_v(\beta, t)< \mathfrak{d}_v$ and $t\neq \beta$.
\end{lemma}

\begin{proof}
It follows from the basic facts about local height functions that $\lambda_{v, D}(t)$ is bounded near such a point, so it suffices to show that $\hat{\lambda}^+_{v, \phi_t}(P_t)$ is, too.  Let $\phi^N(P)=(x_N, y_N)$.  We proceed much as in the proof of Lemma~13 of \cite{variation}.

In particular, if $v$ is a non-archimedean place at which $x, y, a, b_{d-1}, ..., b_0$ are given by Laurent series in a uniformizer $u$, with $v$-adic integral coefficients, then $x_N$ and $y_N$ are given by such series as well.  We have, in this case
\[\left|u_t^{\ord_\beta(y_N)}y_{N, t}\right|_v\leq 1,\]
for all $t$ with $\delta_v(\beta, t)<1$.  If we restrict attention, for the moment, to $t$ satisfying $\delta_v(\beta, t)\geq \delta>0$, then we have
\[\log|y_{N, t}|_v\leq -\ord_\beta(y_N)\log|u_t|_v\leq -\ord_\beta(y_N)\log\delta^{-1}.\]
 Applying the same argument to the $x_N$, we see that
\begin{eqnarray*}
\hat{\lambda}_{v, \phi_t}(P_t)&=&\lim_{N\to\infty}d^{-N}\log^+\|x_{N, t}, y_{N, t}\|_v\\
&\leq& \lim_{N\to\infty}d^{-N}\left(\log^+\|x_N, y_N\|_\beta\cdot\log\delta^{-1}\right)\\
&=&\hat{\lambda}_{\beta, \phi}(P)\log\delta^{-1}\\
&=&0
\end{eqnarray*}
for all $t$ with $\delta<\delta_v(\beta, t)<1$.  But as $\delta$ was arbitrary, this must hold for all $t$ with $\delta_v(\beta, t)<1$, and $\beta\neq t$.

Now, suppose that $v$ is an archimedean place, and let $\epsilon>0$ be chosen such that the Laurent series defining $x, y, a, b_{d-1}, ..., b_0$ converge on the set of $t$ with $\delta_v(\beta, t)\leq \epsilon$.  For any analytic function $g$ on this set, let $\lceil g\rceil_\epsilon$ denote the maximal value of $g(t)$ on the disk $\delta_v(\beta, t)\leq \epsilon$.  Similarly, if $g$ is analytic on the set of $t$ with $\delta_v(\beta, t)=\epsilon$, we let $[g]_\epsilon$ denote the maximum modulus of $g$ on this set.   Then, by the maximum modulus principle,
\[\left|u_t^{\ord_\beta(y_N)}y_{N, t}\right|_v\leq \left\lceil u^{\ord_\beta(y_N)}y_{N}\right\rceil_\epsilon=\left[u^{\ord_\beta(y_N)}y_{N}\right]_\epsilon=\epsilon^{\ord_\beta(y_N)}\left[ y_{N}\right]_\epsilon,\]
for $t$ satisfying $0<\delta_v(\beta, t)<\epsilon$, and hence
\[\left|y_{N, t}\right|_v\leq \left(\frac{\epsilon}{\delta}\right)^{\ord_\beta(y_N)}\left[ y_{N}\right]_\epsilon\]
for $t$ satisfying $\delta<\delta_v(\beta, t)<\epsilon$.  Arguing in the same fashion for $x_{N, t}$, we have
\begin{eqnarray*}
\hat{\lambda}_{v, \phi_t}(P_t)&=&\lim_{N\to\infty}d^{-N}\log^+\|x_{N, t}, y_{N, t}\|_v\\
&\leq &\lim_{N\to\infty}d^{-N}\log^+\max\left\{[x_N]_\epsilon, [ y_N]_\epsilon\right\}\\ && +\lim_{N\to\infty}d^{-N}\left(\log^+\|x_N, y_N\|_\beta\cdot\log\frac{\epsilon}{\delta}\right)\\
&=&\lim_{N\to\infty}d^{-N}\log^+\max\left\{[x_N]_\epsilon, [y_N]_\epsilon\right\},
\end{eqnarray*}
since $\hat{\lambda}_{\beta, \phi}(P)=0$.  But we note that, if $c$ is chosen so that
\[\max\{[ a]_\epsilon, \max_{i}\{[ b_i]_\epsilon\}, 1\}\leq c\]
for all $i$, then
\begin{eqnarray*}
\log[ y_{N+1}]_\epsilon&\leq& \log^+\max\{[ ax_N]_\epsilon,\max_i\{[ b_iy_N^i]_\epsilon\}\}+\log (d+2)\\
 &\leq& d\log^+\max\{[x_N]_\epsilon, [y_N]_\epsilon\}+\log(d+2)c.
 \end{eqnarray*}
 We also have
 \[\log[ x_{N+1}]_\epsilon = \log[ x_{N}]_\epsilon+\log[ a]_\epsilon\leq d\log^+\max\{[ x_N]_\epsilon, [ y_N]_\epsilon\}+\log(d+2)c,\]
 and so
 \[\log^+\max\left\{[x_{N+1}]_\epsilon, [ y_{N+1}]_\epsilon\right\}\leq d\log^+\max\{[x_N]_\epsilon, [y_N]_\epsilon\}+\log(d+2)c.\]
 It now follows from the standard telescoping sum arguments that the limit
 \[\lim_{N\to\infty}d^{-N}\log^+\max\left\{[x_N]_\epsilon, [y_N]_\epsilon\right\}\] appearing above exists and is finite, bounding $\hat{\lambda}_{v, \phi_t}(P_t)$ for $\delta<\delta_v(\beta, t)<\epsilon$.  But the bound does not depend on $\delta$, and so we have a bound on $\hat{\lambda}_{v, \phi_t}(P_t)$ for $0<\delta_v(\beta, t)<\epsilon$, as desired.  The remaining finitely many non-archimedean places are treated similarly, using the non-archimedean maximum principle \cite[p.~318]{robert}
\end{proof}

\begin{lemma}\label{lem:tate lemma bounded part}
Let $\phi$, $P$, and $D^+$ be as above.  Then for any $M_K$-divisor $\mathfrak{d}$ there is an $M_K$-divisor $\mathfrak{f}$ such that for all $t\in C(\overline{K})$ with
\[\delta_v(\beta, t)\geq \mathfrak{d}_v\]
for each $\beta\in Z$, we have
\[\max\left\{\hat{\lambda}^+_{v, \phi_t}(P_t), \lambda_{v, D}(t)\right\}\leq \mathfrak{f}_v.\]
\end{lemma}

\begin{proof}
This follows from Lemma~8 of \cite{variation}.  In particular, there is an $M_K$-divisor $\mathfrak{m}$ such that if $\delta_v(\beta, t)\geq \mathfrak{e}_{\beta, v}$, for each $\beta\in Z$, then
\[|x_t|_v, |y_t|_v, |a_t|_v, |b_{i, t}|_v\leq \mathfrak{m}_v.\]
Using the standard telescoping sum argument, as in Lemma~\ref{lem:local heights} or Lemma~\ref{lem:tate lemma bad case}, this gives a bound on $\hat{\lambda}_{v, \phi_t}(P_t)$ which depends only on $v$.  Moreover, this bound is 0 at any non-archimedean place for which $\mathfrak{m}_v=1$.
\end{proof}

We are now in a position to complete the proof of Theorem~\ref{th:tate}.
\begin{proof}[Proof of Theorem~\ref{th:tate}]
We focus on the relation for $\hat{h}^+_\phi$, first.  By the three lemmas above, if we have $P\in\basin^+_\beta(\phi)\cup\julia_{\beta, \phi}$ for every place $\beta\in C$, then
\begin{eqnarray*}
\left|\hat{h}_{\phi_t}^+(P_t)-h_{D^+}(t)\right|&\leq&\sum_{v\in M_K}n_v\left|\hat{\lambda}_{v, \phi_t}^+(P_t)-\lambda_{v, D}(t)\right|\\
&\leq & \sum_{v\in M_K}n_v\mathfrak{e}_v,
\end{eqnarray*}
a constant, for $t\in C(K)$.  For $t\in C(\overline{K})$, we may use a similar estimate on
\[\hat{h}_{\phi_t}^+(P_t)-h_{D^+}(t)=\frac{1}{\Gal(L/K)}\sum_{\sigma\in\Gal(L/K)}\sum_{v\in M_K}n_v\left(\hat{\lambda}_{v, \phi_{t^\sigma}}^+(P_{t^\sigma})-\lambda_{v, D}(t^\sigma)\right),\]
for any Galois extension $L/K$.

But note that for any $P\in \affine^2(K)$, there exists an $N$ such that $\phi^N(P)$ has the property mentioned above, so we have
\[\hat{h}_{\phi_t}^+(\phi^N(P_t))=h_{D^+(\phi, \phi^N(P))}(t)+O(1),\]
from which the result follows by the linearity of heights, and the relations $\hat{h}_{\phi_t}^+(\phi^N(P_t))=d^N\hat{h}_{\phi_t}^+(P_t)$ and $D^+(\phi, \phi^N(P))=d^ND^+(\phi, P)$.

The symmetric claim for $\hat{h}^-_{\phi_t}(P_t)$ can be proven in an essentially identical manner, after first producing analogous version of Lemmas~\ref{lem:tate lemma main}, \ref{lem:tate lemma bad case}, and~\ref{lem:tate lemma bounded part}.
\end{proof}

\section{Proof of Theorems~\ref{th:silverman} and \ref{th:unlikely}}\label{sec:applications}

It is now a relatively simple matter to prove Theorems~\ref{th:silverman} and \ref{th:unlikely}.
\begin{proof}[Proof of Theorem~\ref{th:silverman}]
Suppose that $\phi(x, y)=(y, x+f(y))$, with $f(y)\in F[z]$, and suppose that $P\in \affine^2(F)$ is not periodic for $\phi$.  Assuming that $\phi$ is not isotrivial, Theorem~\ref{th:benedetto} tells us that $\hat{h}_{\phi}(P)>0$.  In particular, the divisor $D=D_+(\phi, P)+D_-(\phi, P)$, where $D_\pm(\phi, P)$ are the divisors described in Theorem~\ref{th:tate}, is effective and ample.  But by Theorem~\ref{th:tate}, we have
\[\hat{h}_{\phi_t}(P_t)=h_D(t)+O(1),\]
and so the set of $t$ for which $\hat{h}_{\phi_t}(P_t)=0$ is a set of bounded height relative to $D$ (and hence relative to any ample class on $C$).
\end{proof}

\begin{proof}[Proof of Theorem~\ref{th:unlikely}]
Let $X\subseteq C(K)$ be the set of parameters $t$ which are $s$-integral with respect to $\eta$, and such that $\Ocal_{\phi_t}(P_t)=\Ocal_{\phi_t}(Q_t)$, and suppose that $X$ is infinite.  Without loss of generality, we will suppose that there are infinitely many parameters $t$ such that there exists an $m\geq 0$ with $\phi^m_t(P_t)=Q_t$. Note that for each given $m\geq 0$, there can be only finitely many parameters $t\in C(K)$ such that $\phi^m_t(P_t)=Q_t$, unless we have $\phi^m(P)=Q$ on the generic fibre, since the condition $\phi^m_t(P_t)=Q_t$ is described by the vanishing of non-zero functions on $C$.  So there must be arbitrarily large values $m\geq 0$ such that there exists a $t\in X$ with $\phi^m_t(P_t)=Q_t$.  Now write $m=m_1+m_2$, if $\phi^m_t(P_t)=Q_t$, and let $R=\phi^{m_1}_t(P_t)$.  We have
\[\hat{h}_{\phi_t}^+(R)=d^{-m_2}\hat{h}_{\phi_t}^+(Q_t)=d^{-m_2}h_{D^+(\phi, Q)}(t)+O(1)\]
and
\[\hat{h}_{\phi_t}^-(R)=d^{-m_1}\hat{h}_{\phi_t}^-(P_t)=d^{-m_1}h_{D^-(\phi, P)}(t)+O(1).\]
It follows that, for any degree 1 height $h$ on $C$, we have
\[\hat{h}_{\phi_t}(R)\leq d^{-\min\{m_1, m_2\}}\left(\hat{h}^+_{\phi}(Q)+\hat{h}_{\phi}^-(P)\right)h(t)+O\left(h(t)^{1/2}\right)\]
and so, in particular, for any $\delta>0$ we can find infinitely many $t\in X$ such that
\[\hat{h}_{\phi_t}(R)\leq \delta h(t)+O\left(h(t)^{1/2}\right).\]
Note that if $R$ is periodic for $\phi_t$, then so are $P_t$ and $Q_t$, and so by Theorem~\ref{th:silverman}, this happens for only finitely many $X$.  Discounting those, we have for any $\delta>0$ an infinite set of parameters $t\in C(K)$, all $s$-integral with respect to $\eta$, such that
\begin{equation}\label{eq:height upper unlikely}0<\hat{h}_{\phi_t}(R)\leq \delta h(t)+O\left(h(t)^{1/2}\right).\end{equation}
But note that there is a finite set $S$ of primes such that $t$ is $s$-integral with respect to $\eta$ if and only if $b(t)$ is $s+\# S$-integral.  In particular, the existence of infinitely many $t\in C(K)$, $s$-integral with respect to $\eta$, satisfying \eqref{eq:height upper unlikely}, contradicts Theorem~\ref{th:lang} once $\delta$ is small enough.
\end{proof}

\section{Computations and examples}\label{sec:computations}

We close with some computational work around the family
\[\phi(x, y)=(y, x+y^2+b)\]
over $\QQ$, presenting a means of verifying
 Conjecture~\ref{conj} for a specific value of $b$.  Although we focus on $\QQ$, the algorithm is easily modified to work over any number field.  The proof of Proposition~\ref{prop:comp} is essentially a repeated application of this algorithm.
 We note that the computations here are in spirit the same as those in \cite{i-hutz}, although there are some slight differences in the details.  
 
 Our first lemma gives a method for computing a list of possible periods of $\QQ$-rational periodic points for $\phi$, based on the dynamics modulo a prime of good reduction.  Note that this lemma follows from essentially the same argument as a result of Pezda \cite{pezda}, although we present a proof here both for completeness, and because the aforementioned results of Pezda are more general, and do not have conclusions quite as precise as we need for these computations.  Before we state the result we note that it follows from Lemma~\ref{lem:local heights quadratic} that if $b\in \ZZ_p$, and $Q\in \affine^2(\QQ_p)$ is periodic for $\phi(x, y)=(y, x+y^2+b)$, then the coordinates of $Q$ must be $p$-adic integers.  In particular, it makes sense to speak of the image $\widetilde{Q}\in\affine^2(\FF_p)$ of $Q$ modulo $p$.

In general, if $\psi=(F, G):\affine^2\to\affine^2$ is a polynomial map, we let
\[J_{\psi}(x, y)=\left(\begin{matrix}\frac{\partial F}{\partial x} & \frac{\partial F}{\partial y} \\ \frac{\partial G}{\partial x} & \frac{\partial G}{\partial y}\end{matrix}\right)\]
denote the Jacobian of $\psi$ at $(x, y)$.  We then define the \emph{multiplier} of $\psi$ at the $N$-periodic point $Q$  by
\[\Lambda_{N, \psi}(Q)=\prod_{i=0}^{N-1}J_{\psi}(\psi^{N-1-i}(Q)).\]
Note that this is not a well-defined function of the cycle, as in the case of rational maps of $\mathbb{P}^1$, but is a well-defined conjugacy class.  We may speak unambiguously, then of the order of the matrix $\Lambda_{N, \psi}$ for a given cycle.

\begin{lemma}\label{lem:local to global}
Let $p\geq 5$ be prime, let $\phi(x, y)=(y, x+y^2+b)$, where $b\in \ZZ_p$, let $Q\in\affine^2(\QQ_p)$ have period $N$ for $\phi$, and suppose that the image modulo $p$, $\widetilde{Q}\in \affine^2(\FF_p)$,  has period $M$ for $\widetilde{\phi}$.  Then $N=dM$ for some divisor $d\geq 1$ of the order of $\Lambda_{\widetilde{\phi}}(\widetilde{Q})\in\SL_2(\FF_p)$.
\end{lemma}

\begin{proof}
We will first prove something slightly more general.
Suppose that $\psi(x, y)\in\ZZ_p[x,y]^2$, that the point $\mathcal{O}=(0, 0)$ has prime period $N$  under $\psi$, and that the multiplier matrix \[\Lambda_\psi(\mathcal{O})=\prod_{i=0}^{N-1}J_{\psi}(\psi^{N-1-i}(\mathcal{O}))\] is an element of $\SL_2(\ZZ_p)$.  Suppose, further, that the reduction $\widetilde{\psi}$ of $\psi$ modulo $p$ fixes $\widetilde{\mathcal{O}}\in\affine^2(\FF_p)$, and that the multiplier matrix 
$\Lambda_{\widetilde{\psi}}(\widetilde{\mathcal{O}})=J_{\widetilde{\psi}}(\widetilde{\mathcal{O}})$ is the identity matrix.

Now, if $Q=\psi(\mathcal{O})$, then $Q= \mathcal{O}+O(p^e)$, for some largest $e\geq 1$, where $O(p^e)$ denotes an element of $p^e\ZZ_p^2$.  We also have $J_\psi(\mathcal{O})=J=I+pA$, for $I$ the $2\times 2$ identity matrix, and $A\in M_{2\times 2}(\ZZ_p)$.  Now, by Taylor's Theorem, we have
\begin{eqnarray*}
\psi(\mathcal{O})&=&Q\\
\psi^2(\mathcal{O})&=&Q+JQ+O(p^{2e})\\
&\vdots&\\
\psi^m(\mathcal{O})&=&(I+J+J^2+\cdots+J^{m-1})Q+O(p^{2e})=\left(\frac{I-J^m}{I-J}\right)Q+O(p^{2e}).
\end{eqnarray*}
Now, since $\psi^N(\mathcal{O})=\mathcal{O}$, we must have that
\[(I+J+J^2+\cdots+J^{N-1})Q=\left(\frac{I-J^N}{I-J}\right)Q=O(p^{2e}).\]
But if $N\neq p$, then \[(I+J+J^2+\cdots+J^{N-1})Q= NQ+O(p^{e+1})\neq O(p^{2e}).\]
Similarly, since $J=I+pA$, we have
\[\left(\frac{I-J^p}{I-J}\right)=\frac{I-(I+pA)^p}{pA}=\sum_{i=1}^p\binom{p}{i}p^{i-1}A^{i-1}=pI+O(p^2).\]
This shows that $\phi^p(\mathcal{O})=pQ+O(p^{e+2})+O(p^{2e})$, showing that $\mathcal{O}$ does not have period $p$, except perhaps if $e=1$.  In the case $e=1$, a slightly more refined calculation is needed.

If $e=1$ (and $p\neq 2$), we note by the examining the second term in the Taylor expansion that we have
\[\psi(R)=Q+JR+\frac{1}{2}HRR^T+O(p^3),\]
 whenever $R=O(p)$, where $J$ is the Jacobian of $\psi$ at $\mathcal{O}$, and $H$ the $2\times 2\times 2$ Hessian tensor.  The reader can check, by induction, that we have
\[\psi^m(\mathcal{O})=(I+J+J^2+\cdots+J^{m-1})Q+\frac{m(m-1)(2m-1)}{12}HQQ^T+O(p^{3}).\]
In particular, as $p\geq 5$ we have
$\psi^p(\mathcal{O})=pQ+O(p^3)$, and so it is not the case that $\mathcal{O}$ has period $p$ for $\psi$.

To recap, we've shown that if $\mathcal{O}$ is a point of period $N$ for $\psi(x, y)\in\ZZ_p[x, y]^2$, if $\widetilde{\psi}$ fixes $\widetilde{\mathcal{O}}$ modulo $p$, and $J_{\widetilde{\psi}}(\widetilde{\mathcal{O}})$ is the trivial element of $\SL_2(\FF_p)$, then $N$ is not prime.  But if $N>1$, then we may choose a prime $q\mid N$, and apply this result to $\psi^{N/q}$, under which $\mathcal{O}$ has period $q$, to obtain a contradiction.  Since all periodic points are in $\ZZ_p^2$, and since the conditions of the theorem are invariant under a $\ZZ_p$-linear change of variables, it follows that if $Q\in\ZZ_p^2$ is a periodic point for $\psi(x, y)\in\ZZ_p[x, y]^2$, with $\widetilde{Q}$ fixed by $\widetilde{\psi}$, and with $J_{\widetilde{\psi}}(\widetilde{Q})=I$, then it must be that $Q$ is a fixed point for $\psi$.

Now let $\phi(x, y)=(y, x+y^2+b)$, with $b\in \ZZ_p$, and suppose that $Q$ is a point of period $N$ for $\phi$, and $\widetilde{Q}$ is a point of period $M$ for $\widetilde{\psi}$. Clearly we must have $M\mid N$, so we write $N=dM$.  Now, if we set $\chi=\phi^M$, then $\widetilde{Q}$ is a fixed point of $\chi$, and $J_{\widetilde{\chi}}(\widetilde{Q})=\Lambda_{\widetilde{\phi}}(\widetilde{Q})$ by the chain rule.  Now, replacing $\chi$ with $\phi^r$, where $r$ is the order of $\Lambda_{\widetilde{\phi}}(\widetilde{Q})$, we have a periodic point $Q$ for $\psi$ such that $\widetilde{Q}$ is fixed for $\widetilde{\psi}$, and $\Lambda_{\widetilde{\psi}}(\widetilde{Q})$ is the identity.  It follows from the argument above that $\phi^{Mr}(Q)=\psi(Q)=Q$, and so $Q$ is a point of period divisible by $M$, but dividing $Mr$, and so $N=Md$ for some $d\mid r$.
\end{proof}

Lemma~\ref{lem:local to global} is the crux of the algorithm used to verify Proposition~\ref{prop:comp}.  Given a value $b\in\QQ$, and a prime $p\geq 5$ at which $b$ is integral, one can compile a list of periods for $\phi(x, y)=(y, x+y^2+b)$ modulo $p$, and then use Lemma~\ref{lem:local to global} to construct a finite set $S(p, b)\subseteq \ZZ^+$ such that $N\in S(p, b)$ whenever $\phi$ has a $\QQ$-rational periodic point of period $N$.  One might hope, for a given $b\in\QQ$, that we would obtain
\[\bigcap_{5\leq p\leq X}S(b, p)\subseteq \{1, 2, 3, 4, 6, 8\}\]
for $X$ large enough, where we take $S(b, p)=\ZZ^+$ if $p$ is a bad prime.  It turns out that this is too much to ask: the map $\phi(x, y)=(y, y^2-1/4+x)$ has a point of period 2 at $P=(1/2, -1/2)$, and the multiplier of this cycle is
 $\Lambda_\phi(Q)=\left(\begin{matrix}1 & -1 \\ 1 &0\end{matrix}\right)$.  Note that $\Lambda_\phi(Q)$, and its reduction modulo any (odd) prime, has order 6, and so we will have $12\in S(-\frac{1}{4}, p)$ for any prime $p\geq 5$.  Fortunately, there is an alternate means of verifying Conjecture~\ref{conj} in this case.

\begin{lemma}\label{lem:heights used for computations}
Let $b\in\QQ$, let $\phi(x, y)=(y, x+y^2+b)$, and suppose that $\phi$ has a periodic point $Q\in\affine^2(\QQ)$.  Then the denominator of $b$ is a perfect square, and $h(Q)\leq \frac{1}{2}h(b)+\log 3$.
\end{lemma}

Note that if one should like to verify Conjecture~\ref{conj} for all $b\in \QQ$ with $H(b)\leq T$, one potentially has to apply the algorithm above about $\frac{6}{\pi^2}T^2$ times.  The first observation in Lemma~\ref{lem:heights used for computations} reduces this to about $\frac{6}{\pi^2}T^{3/2}$ applications, which is a significant savings.  The second observation in Lemma~\ref{lem:heights used for computations} gives an alternate means of verifying Conjecture~\ref{conj} for a given value $b\in\QQ$ which, while much more costly than the algorithm described above, is guaranteed to provide a conclusive answer.  In the verification of Proposition~\ref{prop:comp}, this alternate method was used to treat parameters $b\in\QQ$ for which the first method failed to verify the conjecture.

\begin{proof}[Proof of Lemma~\ref{lem:heights used for computations}]
Let $p$ be an odd prime, suppose that $\phi$ has a periodic point $(x, y)\in\QQ_p$, and suppose that $|b|_p>1$.  By Lemma~\ref{lem:close to root}, there is a root $\gamma^2=b$ with $|x-\gamma|_p\leq 1$.  But note that $|\gamma|_p>1$, and so it must be the case that
\[p^{-v_p(x)}=|x|_p=|\gamma|_p=|b|_p^{1/2}=p^{-\frac{1}{2}v_p(b)}.\]
In particular, $v_p(b)$ must be even.  In the case $p=2$, we have the same argument, unless $|\gamma|_2\leq 2$, that is, unless $b=2\alpha$ or $b=4\beta$, for $\alpha, \beta$ odd.  In the second case it remains true that $v_2(b)$ is even, so we consider just the first case.  It also, by Lemma~\ref{lem:close to root}, must be true that $|x_N|_2, |y_N|_2\leq 2$ for all $N$.  Note that if $|y_N|_2=2$ for any $N$, then $|y_{N+1}|_2=|x+y_N^2+b|_2=4$, which is a contradiction.  So we must have $|y_N|\leq 1$ for all $N$, and hence $|x_{N}|_2=|y_{N-1}|_2\leq 1$ for all $N$.  But then $|y_{N+1}|_2=|x_N+y_N^2+b|_2=2$, a contradiction.  So it cannot be the case that $|b|_2=2$.  We have shown that $v_p(b)$ is even whenever $v_p(b)<0$, and so the denominator of $b$ is a perfect square.

For the height bound, we simply note that if $P$ is periodic for $\phi$, then $P\in \affine^2(\QQ)\setminus(\basin^+_v(\phi)\cup\basin^-_v(\phi))$  By Lemma~\ref{lem:close to root}, we have
\[\log^+\|P\|_v\leq \frac{1}{2}\log^+|b|+\log (3)_v\]
for every place $v$.  Summing over all places, with the appropriate weights, we obtain our bound.
\end{proof}

The proof of Proposition~\ref{prop:comp} is simply an application of one or the other of these lemmas for every value of $b$ under consideration, and one could presumably extend the computations significantly from what has been done here.

Another approach to building evidence for Conjecture~\ref{conj} would be to fix $N\not\in\{1, 2, 3, 4, 6, 8\}$, and show that there is no $b\in\QQ$ such that $\phi_b(x, y)$ has a $\QQ$-rational point of period $N$.  Pairs $(b, P)$ such that $P$ is a point of period dividing $N$ for $\phi_b$ are parametrized by a curve $\Gamma_N\subseteq\affine^3$, defined by the two equations implicit in $\phi_b^N(P)=P$.  Of course, these curves are not irreducible, as $N\mid M$ implies $\Gamma_N\subseteq \Gamma_M$, but one could restrict attention to the component $\Gamma'_M\subseteq \Gamma_M$ corresponding to examples of exact period $M$.   For instance, a first step in lending more credence to Conjecture~\ref{conj} would be to show that $\Gamma_5'(\QQ)=\emptyset$.  Although the normalization of the projective closure of $\Gamma_5'$ has genus 14, it also admits several quotients, and it is possible that one of these would be amenable to the Chabauty-Coleman method.
 We plan to investigate these curves in a future project.


\begin{thebibliography}{9}

\bibitem{baker}  M.~Baker. A finiteness theorem for canonical heights attached to rational maps over function fields. 
\emph{J.\ Reine Angew.\ Math.} \textbf{626} (2009), pp.~205--233.

\bibitem{benedetto}  R.~L.~Benedetto. Heights and preperiodic points of polynomials over function fields, \emph{Inter.\ Math.\ Res.\ Not.} (2005), no.~62, pp.~3855-Ð3866

\bibitem{call-silv} G.~S.~Call and J.~H.~Silverman. Canonical heights on varieties with morphisms. \emph{Compositio Math.} \textbf{89} (1993), no.~2, pp.~163--205.

\bibitem{denis} L.~Denis.  Points p\'{e}riodiques des automorphismes affines, \emph{J.\ Reine Angew.\ Math.} \textbf{467} (1995), pp.~157--167.



\bibitem{i-hutz} B.~Hutz and P.~Ingram. On {P}oonen's conjecture concerning rational preperiodic points of quadratic maps,  \emph{Rocky Mountain Journal of Mathematics} (to appear).

\bibitem{ads_lang} P.~Ingram. Lower bounds on the canonical height associated to the morphism $\phi(z)= z^d+c$, \emph{Monat.\  Math.} \textbf{157} (2009), pp.~69--89.

\bibitem{variation} P.~Ingram.  Variation of the canonical height for a family of polynomials, preprint (2010).



\bibitem{kawaguchi} S.~Kawaguchi.  Canonical height functions for affine plane automorphisms.  \emph{Math.\ Ann.} \textbf{335} no.~2 (2006), pp.~285--310.

\bibitem{kawaguchi2} S.~Kawaguchi.  Local and global canonical height functions for affine space regular automorphisms, preprint (2009).

\bibitem{lang} S.~Lang.  \emph{Fundamentals of Diophantine Geometry}.  Springer, 1983.

\bibitem{marcello} S.~Marcello.  Sur les propri\'{e}t\'{e}s arithm\'{e}tiques des it\'{e}r\'{e}s d'automorphismes r\'{e}guliers, \emph{C.\ R.\ Acad.\ Sci.\ Paris S\'{e}r.\ I Math.} \textbf{331} no.~1 (2000), pp.~11-16.

\bibitem{pezda} T.~Pezda.  Cycles of polynomial mappings in several variables, \emph{Manuscripta Math.} \textbf{83} (1994), pp.~279--289.

\bibitem{poonen} B.~Poonen. The classification of rational preperiodic points of quadratic polynomials
over $\mathbb{Q}$: a refined conjecture, \emph{Math.\ Z.} \textbf{228} no.~1 (1998), pp.~11--29.

\bibitem{robert} A.~M.~Robert.  \emph{A Course in $p$-adic Analysis}, volume 198 of \emph{Graduate Texts in
Mathematics}. Springer, 2000.

\bibitem{silv_lang}  J.~H.~Silverman. Lower bound for the canonical height on elliptic curves,
\emph{Duke Math.\ J.} \textbf{48} (1981), pp.~633-Ð648.

\bibitem{aec} J.~H.~Silverman.  \emph{The Arithmetic of Elliptic Curves}, volume 106 of \emph{Graduate Texts in
Mathematics}. Springer, 1986.

\bibitem{silverman} J.~H.~Silverman.  Geometric and arithmetic properties of the {H}\'{e}non map, \emph{Math.\ Z.} \textbf{215} no.~2 (1994), pp.~237--250.

\bibitem{ads} J.~H.~Silverman.  \emph{The Arithmetic of Dynamical Systems}, volume 241 of \emph{Graduate Texts in
Mathematics}. Springer, 2007.

\end{thebibliography}
\end{document}